\documentclass{siamart250211}

\usepackage{amsfonts,amsmath,amssymb,bbm}
\usepackage{mathrsfs,mathtools,stmaryrd,wasysym}
\usepackage{xspace,mydef}
\usepackage{esint}
\usepackage{comment}
\usepackage[shortlabels]{enumitem}
\usepackage{tensor}
\DeclareMathAlphabet{\mathpzc}{OT1}{pzc}{m}{it}

\newcommand{\TheTitle}{The spectral fractional Laplacian with measure valued right hand sides: Analysis and Approximation}
\newcommand{\ShortTitle}{Fractional Laplace with Measures}
\newcommand{\TheAuthors}{E.Ot\'arola, A.J.~Salgado}

\headers{\ShortTitle}{\TheAuthors}

\title{\TheTitle}

\author{
  Enrique Ot\'arola\thanks{Departamento de Matem\'atica, Universidad T\'ecnica Federico Santa Mar\'ia, Valpara\'iso, Chile. (\email{enrique.otarola@usm.cl}, \url{http://eotarola.mat.utfsm.cl/})}
\and
  Abner J.~Salgado\thanks{Department of Mathematics, University of Tennessee, Knoxville, TN 37996, USA.
    (\email{asalgad1@utk.edu}, \url{https://math.utk.edu/people/abner-salgado/})}
}

\ifpdf
\hypersetup{
  pdftitle={\TheTitle},
  pdfauthor={\TheAuthors}
}
\fi

\begin{document}

\maketitle

\begin{abstract}
  We consider the spectral definition of the fractional Laplace operator and study a basic linear problem involving this operator and singular forcing. In two dimensions, we introduce an appropriate weak formulation in fractional Sobolev spaces and prove that it is well-posed. As an application of these results, we analyze a pointwise tracking optimal control problem for fractional diffusion. We also develop a finite element scheme for the linear problem using continuous, piecewise linear functions, prove a convergence result in energy norm, and derive an error bound in $L^2(\Omega)$. Finally, we propose a practical scheme based on a diagonalization technique and derive an error bound in $L^2(\Omega)$ using a regularization argument.
\end{abstract}



\begin{center}
  \emph{
    We dedicate this work to the memory of the Prince of Darkness, John Michael ``Ozzy'' Osbourne.
\\
    Thank you for more than five decades of outstanding music!
  }
\end{center}

\section{Introduction}
\label{sec:Into}

The aim of this paper is to study the following boundary value problem for fractional diffusion with a measure-valued right-hand side:
\begin{equation}
\label{eq:TheEqn}
  \Laps u = \mu \text{ in } \Omega,
\end{equation}
where $\Omega \subset \Real^2$ is a bounded, convex polygon, $\Laps$ denotes the fractional Laplace operator in the sense of spectral theory, $s \in (\tfrac12,1)$, and $\mu$ is a Radon measure; see \cref{sec:Notation} for notation.

One motivation for studying \cref{eq:TheEqn} is an optimal control problem with pointwise tracking. Let $\mathcal{D} \subset \Omega$ be a finite set of observable points, $\{ \mathfrak{u}_\vertex \}_{\vertex \in \mathcal{D}} \subset \mathbb{R}$ a set of desired states, and $\alpha >0$ a regularization parameter. We introduce the cost functional
\begin{equation}
\label{eq:J}
 J(\mathfrak{u},q) \coloneqq \frac{1}{2} \sum_{\vertex \in \mathcal{D}} | \mathfrak{u}(\vertex) - \mathfrak{u}_\vertex |^2 + \frac{\alpha}{2} \| q \|^2_{L^2(\Omega)}.
\end{equation}
Given a function $\mathfrak{f}$ and the control bounds $a,b \in \mathbb{R}$, which are such that $-\infty < a < b < \infty$, the pointwise tracking optimal control problem is: Find $\min J(\mathfrak{u},q)$ such that
\begin{equation}
  (-\Delta)^s \mathfrak{u} = \mathfrak{f} + q \text{ in } \Omega,
    \qquad
  q \in \mathbb{Q}_{ad} \coloneqq \left\{ v \in L^2(\Omega): a \leq v(x) \leq b~\mae~x \in \Omega \right\}.
\label{eq:fractional_tracking}
\end{equation}
One of the main difficulties in both the analysis and discretization of this control problem is that the so-called \emph{adjoint problem}, which is essential in the analysis, is a fractional PDE with a \emph{singular} right-hand side, similar to problem \cref{eq:TheEqn}; namely:
\[
 (-\Delta)^s \mathfrak{p} = \sum_{\vertex \in \mathcal{D}} (\mathfrak{u}(\vertex) - \mathfrak{u}_\vertex) \delta_\vertex \text{ in } \Omega.
\]
Here, $\delta_\vertex$ denotes the Dirac delta supported at $\vertex$.

Several papers address the analysis of nonlocal equations with a measure as a right-hand side; see, for example, \cite{MR2779579,MR3217045,MR3339179,MR4026184,MR4510212,MR4505157}. In particular, \cite{MR3339179} considers general nonlinear (possibly degenerate or singular) integro-differential equations. Despite these advances, most available results have been derived for problems involving the \emph{integral} definition of the fractional Laplace operator. To the best of our knowledge, the analysis of problem \cref{eq:TheEqn} is not available in the literature. Regarding discretization, the problem is open for both the \emph{spectral} and \emph{integral} definitions. Therefore, our goal in this work is to develop a suitable weak formulation in fractional Sobolev spaces, analyze this formulation, discretize it using finite elements, and derive convergence results and error bounds.

We organize our presentation as follows. In \cref{sec:Notation} we establish notation, define the spectral fractional Laplacian, and present some of its properties. The analysis of \cref{eq:TheEqn} is carried out in \cref{sec:Analysis}.
In \cref{sec:control} we apply these results to analyze the pointwise tracking optimal control problem \cref{eq:J}--\cref{eq:fractional_tracking}. We establish existence and uniqueness of an
optimal solution and derive optimality conditions. With the results of \cref{sec:Analysis} at hand, \cref{sec:FEM} describes the discrete framework we shall adopt for the numerical approximation of the linear problem \cref{eq:TheEqn}. In particular, we define the discrete Laplacian, its fractional powers, and recall some norms on discrete spaces that arise from it. An ideal scheme, \ie one that fully mimics the continuous framework is developed and analyzed in \cref{sec:DiscreteProblemVerA}, where we show that this method converges at an optimal rate in $L^2(\Omega)$. However, this method is not amenable to practical implementation. For this reason, in \cref{sec:Gordito}, we present a practical scheme and show that it converges, again in $L^2(\Omega)$, with the same rate as our ideal scheme does. Some qualitative numerical illustrations are presented in \cref{sec:NumExp}.

\section{Notation and preliminary remarks}
\label{sec:Notation}

We begin by introducing some relations that we will use in our work. $A \coloneqq B$ denotes equality by definition. $C \eqqcolon D$ stands for $D \coloneqq C$. $A \lesssim B$ means $A \leq c B$ for a nonessential constant $c$ that may change at each occurrence. $A \gtrsim B$ means $B \lesssim A$. Finally, $A \eqsim B$ is the short form for $A \lesssim B \lesssim A$.

Let $\Omega \subset \Real^2$ be a bounded, convex polygon. Throughout the text, we use standard notation for classical Lebesgue and Sobolev spaces. The space of finite Radon measures on $\Omega$ is denoted by $\calM(\Omega)$; see \cite[Definition 1.9]{MR3409135}. The duality pairing between $\calM(\Omega)$ and $C_0(\bar\Omega)$ --- the space of continuous functions in $\bar \Omega$ vanishing on $\partial \Omega$ --- will be denoted by $\langle\cdot,\cdot\rangle$.

With respect to our problem data, we assume that $s \in (\tfrac12,1)$ and that $\mu \in \calM(\Omega)$.

\subsection{The spectral fractional Laplacian}
\label{sub:FracLaplace}

We now briefly describe the construction of the spectral fractional Laplacian \cite{MR2646117,MR2825595,MR2754080}. For further details, we refer the reader to \cite{MR3348172,MR3989717,MR3893441}. 

The eigenvalue problem: Find $(\lambda,\varphi) \in \Real \times H_0^1(\Omega) \setminus \{ 0\}$ such that
\begin{equation}
\label{eq:LapEigenPairs}
  (\nabla \varphi, \nabla v)_{L^2(\Omega)} = \lambda (\varphi,v)_{L^2(\Omega)}
  \quad
  \forall v \in H_0^1(\Omega)
\end{equation}
has a countable collection of solutions $\{(\lambda_k,\varphi_k)\}_{k=1}^\infty \subset \Real^+ \times \Hunz$ such that $\{\varphi_k\}_{k=1}^\infty$ is an orthonormal basis of $L^2(\Omega)$ and an orthogonal basis of $\Hunz$ \cite{MR609148}.

For $r \geq 0$ we define, in terms of the sequence of eigenpairs $\{(\lambda_k,\varphi_k)\}_{k=1}^\infty$,
\begin{equation}
\label{eq:DefOfPolHr}
  \polH^r(\Omega) \coloneqq \left\{ w = \sum_{k=1}^\infty w_k \varphi_k \ : \ \sum_{k=1}^\infty \lambda_k^r |w_k|^2 < \infty \right\},
\end{equation}
with norm
\begin{equation}
  \| w\|_{\polH^r(\Omega)} \coloneqq \left( \sum_{k=1}^\infty \lambda_k^r |w_k|^2 \right)^{\frac{1}{2}}.
  \label{eq:DefofPolHrnorm}
\end{equation}
For $r >0$, $\polH^{-r}(\Omega)$ is the dual space of $\polH^{r}(\Omega)$. The duality pairing between $\mathbb{H}^{-r}(\Omega)$ and $\mathbb{H}^{r}(\Omega)$ is denoted by $_{-r}\langle \cdot, \cdot \rangle_{r}$. With $_{-r}\langle \cdot, \cdot \rangle_{r}$ we can extend the definition of the norm in \cref{eq:DefofPolHrnorm} to negative values of $r$. In fact, through this duality pairing, we can identify an element $f$ of $\mathbb{H}^{-r}(\Omega)$ with a sequence $\{ f_k \}_{k=1}^{\infty}$ such that
\[
 \sum \lambda_k^{-r}f_k^2 \eqqcolon \| f\|^2_{\polH^{-r}(\Omega)} < \infty.
\]

For $s \in (0,1)$ and $w \in C_0^\infty(\Omega)$, we thus define the \emph{spectral fractional Laplacian} as \cite{MR2646117,MR2825595,MR2754080}
\begin{equation}
\label{eq:DeofOfLaps}
  \Laps w \coloneqq \sum_{k=1}^\infty \lambda_k^s w_k \varphi_k, \qquad w_k \coloneqq \int_\Omega w \varphi_k \diff x.
\end{equation}
The operator $\Laps$ can be extended to $\mathbb{H}^s(\Omega)$ by density: $\Laps: \mathbb{H}^s(\Omega) \rightarrow \mathbb{H}^{-s}(\Omega)$. We note that $\Laps$ is an isomorphism between $\polH^s(\Omega)$ and its dual space $\polH^{-s}(\Omega)$.

Given the definition of the spectral fractional Laplacian in terms of $\{(\lambda_k,\varphi_k)\}_{k=1}^\infty$, the spaces $\polH^r(\Omega)$ are natural for problems involving this definition. However, in the following analysis, we will need the relationship between $\polH^r(\Omega)$ and the classical fractional Sobolev spaces. To address this, we note that the spaces $\polH^r(\Omega)$ can also be obtained as intermediate spaces in the sense of \cite[Chapter 1]{MR350177}; namely, we have
\[
  \polH^r(\Omega) = D( (-\Delta)^{r/2} ).
\]
For $r \geq 0$ these spaces are Hilbert when endowed with the norm
\[
  \| w \|_{r} \coloneqq \| (-\Delta)^{r/2} w\|_{L^2(\Omega)} ;
\]
see also \cite[Remark 7.6]{MR350177}. From this, the following characterization of the spaces $\polH^r(\Omega)$ can be derived; see \cite{MR350177,MR1742312,MR2328004,MR216336,MR3343061,MR3356020} for details.

\begin{proposition}[characterization of $\mathbb{H}^r(\Omega)$ for $r \in (0,2)$]
\label{pro:characterization}
We have that 
\[
   \polH^r(\Omega) = \begin{dcases}
                       H^r(\Omega), & r \in \left[ 0,\tfrac12 \right),\\
                       H^{\frac{1}{2}}_{00}(\Omega), & r = \tfrac12, \\
                       H_0^r(\Omega), & r \in \left(\tfrac12,1\right],
                     \end{dcases}
 \]
with equivalent norms. Moreover, if $r \in (1,2)$, then $\polH^r(\Omega) = H_0^1(\Omega) \cap H^r(\Omega)$, with equivalent norms.
\end{proposition}

For future reference, it should be mentioned that the inverse of $(-\Delta)^s$ can be given by the so-called Balakrishnan formula; see \cite[eq. (4.4)]{MR3753604} or \cite[Section 10.4]{MR609148}
\begin{equation}
\label{eq:DefOfBalakrishnan}
  (-\Delta)^{-s} = \frac{2\sin(\pi s)}\pi \int_0^\infty t^{1-2s}(t^2 I - \Delta)^{-1} \diff t.
\end{equation}

\section{Analysis of the problem}
\label{sec:Analysis}

We now begin the analysis of problem \cref{eq:TheEqn}. For this purpose, given $s \in (\tfrac12,1)$, we choose $\theta \in (1-s,s)$ and define
\begin{equation}
 \label{eq:mathcalA}
  \calA : \polH^{s-\theta}(\Omega) \times \polH^{s+\theta}(\Omega) \to \Real,
  \qquad
  (v,w) \mapsto \calA(v,w) \coloneqq \sum_{k=1}^\infty \lambda_k^s v_k w_k,
\end{equation}
where
\[
  v = \sum_{k=1}^\infty v_k \varphi_k, \qquad w = \sum_{k=1}^\infty w_k \varphi_k.
\]
It is clear that the parameters $s$ and $\theta$ satisfy the following important inequalities:
\begin{equation}
\label{eq:s_and_theta}
 s - \theta > 0,
 \qquad
 1 < s + \theta < 2s < 2.
\end{equation}

Having defined the form $\mathcal{A}$, we propose the following weak formulation for problem \cref{eq:TheEqn}: Find $u \in \polH^{s-\theta}(\Omega)$ such that
\begin{equation}
\label{eq:WeakFormulation}
\calA(u,v) = \langle \mu, v \rangle, \qquad \forall v \in \polH^{s+\theta}(\Omega).
\end{equation}

The following remark is now in order.
\begin{remark}[on the role of $s$ and $\theta$ in \cref{eq:WeakFormulation}]
\label{rk:role_s_theta}
  Note that $s - \theta > 0$ implies we seek for a solution in a fractional Sobolev space with differentiability index $s - \theta >0$: $\mathbb{H}^{s-\theta}(\Omega)$. Second, since $1< s + \theta < 2s$, we have that $H^{s+\theta-1}(\Omega) \hookrightarrow L^{p}(\Omega)$ for every $p \leq p^\star = 2/(2-s-\theta)$ \cite[Theorem 6.7]{MR2944369} and thus that $H^{s+\theta}(\Omega) \hookrightarrow W^{1,p^\star}(\Omega) \hookrightarrow C(\bar \Omega)$ \cite[Theorem 4.12, Part II]{MR2424078} because $p^\star>2$. Consequently, $\mu$ defines a bounded linear functional on $\mathbb{H}^{s+\theta}(\Omega)$, and the right-hand side of problem \cref{eq:WeakFormulation} is well defined.
\end{remark}

To prove the well posedness of problem \cref{eq:WeakFormulation}, it suffices to show that the conditions of the so-called BNB theorem \cite[Theorem 25.9]{MR4269305} are satisfied.

\begin{theorem}[BNB]
\label{thm:BNB}
  The bilinear form $\calA$ is bounded and satisfies 
  \begin{align}
  \label{eq:inf_sup_A}
    \inf_{v \in \polH^{s-\theta}(\Omega)}\sup_{w \in \polH^{s+\theta}(\Omega)} \frac{ \calA(v,w) }{ \| v \|_{\polH^{s-\theta}(\Omega)} \| w \|_{\polH^{s+\theta}(\Omega)} } & \geq 1 
    \\
    \inf_{w \in \polH^{s+\theta}(\Omega)}\sup_{v \in \polH^{s-\theta}(\Omega)} \frac{ \calA(v,w) }{ \| v \|_{\polH^{s-\theta}(\Omega)} \| w \|_{\polH^{s+\theta}(\Omega)} } & \geq 1.
  \end{align}
  Consequently, for every $\mu \in \calM(\Omega)$, problem \cref{eq:WeakFormulation} has a unique solution $u \in \polH^{s-\theta}(\Omega)$ that satisfies the stability bound
  \begin{equation}
  \label{eq:continuous_stability}
    \| u \|_{\polH^{s-\theta}(\Omega)} \lesssim \| \mu \|_{\calM(\Omega)},
  \end{equation}
  where the implicit constant depends only on $s$, $\theta$, and $\Omega$.
\end{theorem}
\begin{proof}
The boundedness of the bilinear form $\calA$ is nothing but an observation. In fact, if $v \in \mathbb{H}^{s-\theta}(\Omega)$ and $w \in \mathbb{H}^{s+\theta}(\Omega)$, then
  \begin{align*}
    |\calA(v,w)| &=
    \left| \sum_{k=1}^\infty \lambda_k^{s} v_k w_k \right|
    =
    \left| \sum_{k=1}^\infty \lambda_k^{(s-\theta)/2} v_k \lambda_k^{(s+\theta)/2} w_k \right|
    \\
      &\leq \left( \sum_{k=1}^\infty \lambda_k^{s-\theta} |v_k|^2 \right)^{\frac{1}{2}} \left( \sum_{k=1}^\infty \lambda_k^{s+\theta} |w_k|^2 \right)^{\frac{1}{2}}
      = \| v \|_{\polH^{s-\theta}(\Omega)} \| w \|_{\polH^{s+\theta}(\Omega)}.
  \end{align*}
  
  We now prove the inf-sup condition
  \begin{equation}
  \label{eq:first_inf_sup_A}
  \inf_{v \in \polH^{s-\theta}(\Omega)}\sup_{w \in \polH^{s+\theta}(\Omega)} \frac{ \calA(v,w) }{ \| v \|_{\polH^{s-\theta}(\Omega)} \| w \|_{\polH^{s+\theta}(\Omega)} } \geq 1.
  \end{equation}
  To do this, we proceed as follows. Given $v = \sum_{k=1}^\infty v_k \varphi_k \in \polH^{s-\theta}(\Omega)$, we define
  \[
  w_k \coloneqq \lambda_k^{-\theta} v_k,
  \qquad
    w_v \coloneqq \sum_{k=1}^\infty w_k \varphi_k.
  \]
  With this definition it follows that $\calA(v,w_v) = \| v \|_{\polH^{s-\theta}(\Omega)}^2$. In fact,
  \[
    \calA(v,w_v) = \sum_{k=1}^\infty \lambda_k^s v_k \left( \lambda_k^{-\theta} v_k \right) = \sum_{k=1}^\infty \lambda_k^{s-\theta} |v_k|^2 = \| v \|_{\polH^{s-\theta}(\Omega)}^2.
  \]
  In addition, we have that
  \[
    \| w_v \|_{\polH^{s+\theta}(\Omega)}^2 = \sum_{k=1}^\infty \lambda_k^{s+\theta} |w_k|^2 = \sum_{k=1}^\infty \lambda_k^{s+\theta} |\lambda_k^{-\theta} v_k|^2 = \sum_{k=1}^\infty \lambda_k^{s-\theta} |v_k|^2 = \| v \|_{\polH^{s-\theta}(\Omega)}^2.
  \]
  As a result, we obtain
  \[
   \sup_{w \in \polH^{s+\theta}(\Omega)} \frac{ \calA(v,w) }{ \| v \|_{\polH^{s-\theta}(\Omega)} \| w \|_{\polH^{s+\theta}(\Omega)} }
   \geq
 \frac{ \calA(v,w_v) }{ \| v \|_{\polH^{s-\theta}(\Omega)} \| w_v \|_{\polH^{s+\theta}(\Omega)} } = 1,
  \]
  which implies the inf-sup condition \cref{eq:first_inf_sup_A}.

  The proof of the inf-sup condition
  \begin{equation}
  \label{eq:second_inf_sup_A}
    \inf_{w \in \polH^{s+\theta}(\Omega)}\sup_{v \in \polH^{s-\theta}(\Omega)} \frac{ \calA(v,w) }{ \| v \|_{\polH^{s-\theta}(\Omega)} \| w \|_{\polH^{s+\theta}(\Omega)} } \geq 1
  \end{equation}
  is similar and it is omitted for the sake of brevity.
  
  With the inf-sup conditions \cref{eq:first_inf_sup_A} and \cref{eq:second_inf_sup_A} established, we have shown that $\calA$ satisfies all the conditions of the BNB theorem \cite[Theorem 25.9]{MR4269305}.

  Finally, since $1< s + \theta < 2s$, we have that
  \[
   \| \mu \|_{\mathbb{H}^{-s-\theta}(\Omega)} = \sup_{w \in \polH^{s+\theta}(\Omega)} \frac{\langle \mu, w \rangle}{ \| w \|_{\polH^{s+\theta}(\Omega)}} 
   \leq 
   \| \mu \|_{\mathcal{M}(\Omega)} \sup_{w \in \polH^{s+\theta}(\Omega)} \frac{ \| w \|_{C(\bar \Omega)}}{ \| w \|_{\polH^{s+\theta}(\Omega)}} 
   \lesssim \| \mu \|_{\mathcal{M}(\Omega)},
  \]
  where we have used that $\| w \|_{C(\bar \Omega)} \leq C \| w \|_{\polH^{s+\theta}(\Omega)}$ with a constant $C$ that depends on $s$, $\theta$, and $\Omega$; see \cref{rk:role_s_theta} for details.

  The desired result follows from a direct application of the BNB theorem \cite[Theorem 25.9]{MR4269305}, \cite[Theorem 2.2]{NV}. This concludes the proof.
\end{proof}

\begin{remark}[scaling]
\label{rem:Delta} 
To gain intuition, we consider the case of $\mu = \delta_{\vertex}$, the Dirac measure supported at the point $\vertex \in \Omega$. In this case, as shown in \cite[Theorem 2.7]{MR3489634}, the fundamental solution (Green's function) behaves as
  \[
    \calG(x,\vertex) \eqsim |x-\vertex|^{-2+2s}, \qquad x,\vertex \in \Omega, \quad x \neq \vertex.
  \]
  We note that, provided $p < \tfrac1{1-s}$, we have
  \[
    \int_\Omega |\calG(x,\vertex)|^p \diff x \eqsim \int_0^{\diam \Omega} r^{p(-2+2s)} r\diff r < \infty .
  \]
  Let now $\bfbeta$ be a multi-index. Then,
  \[
    \partial^\bfbeta_x \calG(x,\vertex) \eqsim |x - \vertex |^{-2+2s - |\bfbeta|}, \qquad x,\vertex \in \Omega, \quad x \neq \vertex.
  \]
  This implies that, provided $|\bfbeta| < 2s-1$, we have
  \[
    \int_\Omega |\partial_x^\bfbeta \calG(x,\vertex)|^2 \diff x \eqsim \int_0^{\diam\Omega} r^{2(-2+2s-|\bfbeta|)} r \diff r < \infty,
  \]
  \ie $\partial^\bfbeta_x \calG \in \Ldeux$. Thus, we may want to set, at least formally, $|\bfbeta| = s-\theta$ so that $\calG(\cdot,\vertex) \in \polH^{s-\theta}(\Omega)$. However, for this to be possible, we must require
  \[
    s-\theta < 2s-1 \qquad \implies \qquad s+\theta > 1.
  \]
  This is consistent with the assumptions we have imposed on $s$ and $\theta$; see the inequalities in \cref{eq:s_and_theta} and \cref{rk:role_s_theta}.
\end{remark}

\section{The pointwise tracking optimal control problem}
\label{sec:control}

We now apply the results we have obtained so far, and precisely describe the pointwise tracking optimal control problem introduced in \cref{sec:Into}. We establish existence and uniqueness of an optimal solution and derive first-order necessary and sufficient optimality conditions.

We recall that $\Omega \subset \Real^2$ is a bounded, convex polygon, and that $s \in (\tfrac12,1)$. We begin our analysis by introducing the bilinear form
\begin{equation}
 \label{eq:mathcalB}
  \calB : \polH^{s}(\Omega) \times \polH^{s}(\Omega) \to \Real,
  \quad
  (v,w) \mapsto \calB(v,w) \coloneqq \tensor[_{-s}]{\langle \Laps v, w \rangle}{_{s}} = \sum_{k=1}^\infty \lambda_k^s v_k w_k.
\end{equation}
Given $\mathfrak{f} \in L^2(\Omega)$, we formulate the pointwise tracking optimal control problem as follows: Find
\begin{equation}
\label{eq:min}
  \min \left\{ J(\mathfrak{u},q) \ : \ (\mathfrak{u},q) \in \mathbb{H}^s(\Omega) \times \mathbb{Q}_{ad} \right\},
\end{equation}
subject to the state equation
\begin{equation}
\label{eq:state_equation}
  \mathfrak{u} \in \mathbb{H}^s(\Omega):
  \quad
  \mathcal{B}(\mathfrak{u},v) = \int_\Omega \left( \mathfrak{f} + q \right)  v \diff{x}
  \quad
  \forall v \in \mathbb{H}^s(\Omega).
\end{equation}
We recall that the functional $J$ is defined in \cref{eq:J} and the set $\mathbb{Q}_{ad} \subset L^2(\Omega)$ is defined in \cref{eq:fractional_tracking}. The control bounds $a,b \in \Real$ verify $-\infty < a < b < \infty$. A straightforward application of the Lax-Milgram lemma shows that, for every $q \in \polQ_{ad}$, there is a unique solution $\mathfrak{u} \in \mathbb{H}^{s}(\Omega)$ to problem \eqref{eq:state_equation}. In addition, since $\mathfrak{f} + q \in L^2(\Omega)$, it follows directly from \cref{eq:DefOfPolHr} that $\mathfrak{u} \in \mathbb{H}^{2s}(\Omega)$. Therefore, by invoking \cref{pro:characterization}, the fact that $s>\tfrac{1}{2}$, and the Sobolev embedding $H^{2s}(\Omega) \hookrightarrow C(\bar \Omega)$; see \cite[Theorem 6.7]{MR2944369} and \cite[Theorem 4.12, \textbf{Part II}]{MR2424078}, we deduce that point evaluations of $\mathfrak{u}$ are well-defined and, consequently, so is the cost functional.

The next ingredient in the analysis is to introduce the \emph{control-to-state map}
\[
  \mathcal{S}: L^2(\Omega) \rightarrow \mathbb{H}^s(\Omega) \cap C(\bar \Omega),
  \qquad
  q \mapsto \fraku \coloneqq \calS q,
\]
which maps $q \in L^2(\Omega)$ to the unique $\mathfrak{u}\in \mathbb{H}^s(\Omega) \cap C(\bar \Omega)$ that solves \cref{eq:state_equation}. We emphasize that, since $s > \frac12$, the Sobolev embedding $H^{2s}(\Omega) \hookrightarrow C(\bar \Omega)$ guarantees that $\mathcal{S}$ is well-defined. Moreover, $\mathcal{S}$ is an affine and continuous map; specifically,
\begin{equation}
\label{eq:u_is_continuous}
  \| \mathfrak{u} \|_{\mathbb{H}^s(\Omega)}
  +
  \| \mathfrak{u} \|_{C(\bar \Omega)}
  \lesssim
  \| \mathfrak{f} \|_{L^2(\Omega)}
  +
  \| q \|_{L^2(\Omega)}.
\end{equation}
With this operator at hand, we define the \emph{reduced cost functional}
\begin{equation}
\label{eq:reduced_functional}
j: \mathbb{Q}_{ad} \rightarrow \Real,
\quad
q \mapsto j(q) \coloneqq J(\mathcal{S}q,q) = \frac{1}{2} \sum_{\vertex \in \mathcal{D}} | \mathcal{S}q(\vertex) - \mathfrak{u}_\vertex |^2 + \frac{\alpha}{2} \| q \|^2_{L^2(\Omega)}.
\end{equation}

Let us now show the existence and uniqueness of an optimal control.

\begin{theorem}[existence and uniqueness]
  The pointwise tracking optimal control problem \cref{eq:min}--\cref{eq:state_equation} has a unique solution $(\bar{\mathfrak{u}},\bar{q}) \in \mathbb{H}^s(\Omega) \times \mathbb{Q}_{ad}$.
\end{theorem}
\begin{proof}
  The set $\mathbb{Q}_{ad} \subset L^2(\Omega)$ is nonempty, closed, bounded, and convex. Therefore it is weakly sequentially compact in $L^2(\Omega)$. Next we observe that the reduced cost functional $j$ is continuous and, since $\alpha >0$, it is strictly convex. Consequently, $j$ is weakly lower semicontinuous. We now note that the Sobolev embedding $H^{2s}(\Omega) \hookrightarrow C(\bar \Omega)$ is not only continuous but also compact \cite[Corollary 7.2]{MR2944369}. The assertion thus follows by applying the direct method of the calculus of variations, as in the proof of \cite[Theorem 2.14]{MR2583281}. Uniqueness follows from the strict convexity of $j$.
\end{proof}

The following result is classical \cite[Lemma 2.21]{MR2583281}: the control $\bar{q} \in \mathbb{Q}_{ad}$ is optimal for problem \cref{eq:min}--\cref{eq:state_equation} if and only if
\begin{equation}
\label{eq:VI}
  j'(\bar{q})(q - \bar q) \geq 0 \qquad \forall q \in \mathbb{Q}_{ad}.
\end{equation}
To explore this variational inequality and to obtain first-order optimality conditions, we introduce the so-called \emph{adjoint problem}: Find $\mathfrak{p} \in \mathbb{H}^{s-\theta}(\Omega)$ such that
\begin{equation}
\label{eq:adjoint_equation}
 \mathcal{A}(\mathfrak{p},v) = \sum_{\vertex \in \mathcal{D}} \left\langle (\mathfrak{u}(\vertex) - \mathfrak{u}_\vertex) \delta_\vertex, v \right \rangle
 \quad
 \forall v \in \mathbb{H}^{s+\theta}(\Omega).
\end{equation}
Here, as before, $\theta \in (1-s,s)$. An immediate application of \cref{thm:BNB} yields the existence and uniqueness of the \emph{adjoint state} $\mathfrak{p}$, along with its stability bound:
\begin{equation}
\label{eq:p_continuous_stability}
  \| \mathfrak{p} \|_{\polH^{s-\theta}(\Omega)} \lesssim \sum_{\vertex \in \mathcal{D}} \| \mathfrak{u}(\vertex) - \mathfrak{u}_\vertex \|_{C(\bar \Omega)} \lesssim \| \mathfrak{f} \|_{L^2(\Omega)} + \| q \|_{L^2(\Omega)} + \sum_{\vertex \in \mathcal{D}} |\mathfrak{u}_\vertex|,
\end{equation}
where we have used the stability bound \cref{eq:u_is_continuous}.

We are now in a position to show first-order optimality conditions for our problem.

\begin{theorem}[optimality conditions]
 The control $\bar{q} \in \mathbb{Q}_{ad}$ is optimal for the pointwise tracking optimal control problem \cref{eq:min}---\cref{eq:state_equation} if and only if
 \begin{equation}
  \label{eq:first_order_optimality_condition}
   \int_\Omega \left( \bar{\mathfrak{p}} + \alpha \bar{q}\right) \left( q - \bar{q} \right) \diff{x}  \geq 0 \qquad \forall q \in \mathbb{Q}_{ad}.
 \end{equation}
  Here, $\bar{\mathfrak{p}} \in \polH^{s-\theta}(\Omega)$ is the unique solution to \cref{eq:adjoint_equation} with $\bar \fraku = \calS \bar{q}$.
\end{theorem}
\begin{proof}
  From \cref{eq:VI} we immediately deduce that, for every $q \in \polQ_{ad}$,
  \begin{equation}
  \label{eq:VI2}
    0 \leq j'(\bar{q})(q - \bar q) = \sum_{\vertex \in \mathcal{D}} (\mathcal{S}\bar{q}(\vertex) - \mathfrak{u}_\vertex) \mathcal{S} (q - \bar q)(\vertex)  + \alpha \int_\Omega \bar{q} ( q -\bar{q}) \diff{x}.
  \end{equation}
  The second term on the right hand side of the previous expression is already present in \cref{eq:first_order_optimality_condition}, so we focus on the first term. Let $q \in \mathbb{Q}_{ad}$ and set $\mathfrak{u} = \mathcal{S} q$ and $\bar{\mathfrak{u}} = \mathcal{S} \bar{q}$.
  From \cref{eq:DefOfPolHr} we deduce that $\mathfrak{u} , \bar{\mathfrak{u}} \in \mathbb{H}^{2s}(\Omega)$. Next, since $s+\theta < 2s$, we have $\mathbb{H}^{2s}(\Omega) \hookrightarrow \mathbb{H}^{s + \theta}(\Omega)$. As a result, $v = \mathfrak{u} - \bar{\mathfrak{u}} \in \polH^{s+\theta}(\Omega)$ is an admissible test function in \cref{eq:adjoint_equation}. Thus,
  \begin{equation}
  \label{eq:first_order_aux_1}
    \mathcal{A}(\bar{\mathfrak{p}}, \mathfrak{u} - \bar{\mathfrak{u}}) = \sum_{\vertex \in \mathcal{D}} \left\langle (\bar{\mathfrak{u}}(\vertex) - \mathfrak{u}_\vertex) \delta_\vertex, \mathfrak{u} - \bar{\mathfrak{u}} \right \rangle
    =  \sum_{\vertex \in \mathcal{D}} (\bar{\mathfrak{u}}(\vertex) - \mathfrak{u}_\vertex) (\mathfrak{u}(\vertex) - \bar{\mathfrak{u}}(\vertex)).
  \end{equation}
  On the other hand, we would like to set $v = \mathfrak{p} \in \mathbb{H}^{s-\theta}(\Omega)$ in the problem that $\mathfrak{u} - \bar{\mathfrak{u}}$ solves. If that were possible, we would obtain
  \begin{equation}
    \calB(\fraku - \bar{\fraku}, \bar{\frakp}) =  \int_\Omega \left( q - \bar{q} \right) \bar{\mathfrak{p}} \diff{x}.
  \label{eq:first_order_aux_2}
  \end{equation}
  However, $\mathfrak{p} \in \mathbb{H}^{s-\theta}(\Omega) \setminus \mathbb{H}^{s}(\Omega)$, so  \cref{eq:first_order_aux_2} must be justified with a different argument. To achieve this, we let $\{ p_n \}_{n \in \mathbb{N}} \subset C_0^{\infty}(\Omega)$ be such that $p_n \rightarrow \bar{\mathfrak{p}}$ in $\mathbb{H}^{s-\theta}(\Omega)$. Since, for every $n \in \polN$, we have $p_n \in \mathbb{H}^s(\Omega)$, we can set $v = p_n$ in the problem that $\mathfrak{u} - \bar{\mathfrak{u}}$ solves to obtain that
  \[
   \mathcal{B}(\mathfrak{u} - \bar{\mathfrak{u}}, p_n) =  \int_\Omega \left(q - \bar{q}\right) p_n \diff{x} \quad \forall n \in \mathbb{N}.
  \]
  Since $\mathfrak{u} - \bar{\mathfrak{u}} \in \mathbb{H}^{2s}(\Omega) \hookrightarrow \mathbb{H}^{s + \theta}(\Omega)$, we have that
  \[
    \calB(\fraku - \bar{\fraku}, p_n ) = \calA( p_n, \fraku - \bar{\fraku} ).
  \]
  We now invoke the continuity of $\mathcal{A}$ to obtain
  \[
    \calA( \bar{\mathfrak{p}}, \fraku - \bar{\fraku} ) = \int_\Omega \left(q - \bar{q}\right) \bar{\mathfrak{p}} \diff{x}.
  \]
  From this and identity \cref{eq:first_order_aux_1}, we conclude that
  \[
   \int_\Omega \left(q - \bar{q}\right) \bar{\mathfrak{p}} \diff{x}
   =
   \sum_{\vertex \in \mathcal{D}} (\bar{\mathfrak{u}}(\vertex) - \mathfrak{u}_\vertex) (\mathfrak{u}(\vertex) - \bar{\mathfrak{u}}(\vertex)).
  \]
  The desired variational inequality then follows from \cref{eq:VI2}.
\end{proof}

Define the projection operator
\[
  \Pi_{[a,b]} : L^1(\Omega) \to \mathbb{Q}_{ad}, \qquad \Pi_{[a,b]}(v) = \min\{b,\max\{ a,v\}\}.
\] 
Following \cite[Section 2.8.2]{MR2583281} we derive that $\bar q$ solves \cref{eq:first_order_optimality_condition} if and only if
\[
 \bar{q} = \Pi_{[a,b]}\left(-\alpha^{-1} \bar{\mathfrak{p}} \right), \quad \mae \text{ in } \Omega.
\]
An immediate application of \cite[Theorem 1]{MR1173747} shows that $\bar{q} \in \mathbb{H}^{s-\theta}(\Omega)$.

\section{Discretization}
\label{sec:FEM}

We will perform the discretization using finite elements. Since $\Omega$ is a polygon, it can be meshed exactly. We thus introduce $\polT = \{\Triang\}_{h>0}$, a quasiuniform family of conforming triangular meshes of $\bar\Omega$. The parameter $h>0$ denotes the mesh size of $\Triang$. By $\Fespace$, we denote the space of continuous functions that are piecewise linear with respect to the mesh $\Triang$ and vanish on $\partial\Omega$. We note that $\Fespace$ satisfies 
$\Fespace \subset \polH^r(\Omega)$ for all $r \in [0,\tfrac{3}{2})$. Given $h>0$, we denote
\[
  N_h = \dim\Fespace.
\]
We also introduce $\{\phi_n\}_{n=1}^{N_h}$, the canonical nodal basis of $\Fespace$. As a final ingredient, we introduce the dual basis $\{\phi_n^\star \}_{n=1}^{N_h}$ of $\Fespace$, which satisfies
\begin{equation}
 \int_{\Omega} \phi_n \phi_m^\star \diff{x} = \delta_{n,m}, \qquad n,m = 1, \ldots, N_h.
 \label{eq:dual_basis}
\end{equation}

In our constructions, we will need an interpolant $\calI_h$ that is stable in $\polH^r(\Omega)$ and has suitable approximation properties, namely,
\begin{align}
    \label{eq:Interpolant_1}
    \| \calI_h w \|_{\polH^r(\Omega)} &\lesssim \| w \|_{\polH^r(\Omega)}, & r &\in \left[0,\frac32 \right),
    \\
    \label{eq:Interpolant_2}
    \| w - \calI_h w \|_{\polH^t(\Omega)} &\lesssim h^{m-t} \| w \|_{\polH^m(\Omega)}, & m &\in [0,2], \quad t \in \left[ 0,\min\left\{\frac32, m \right\} \right].
\end{align}
A suitable choice is the so-called Scott-Zhang interpolant \cite{MR3118443,MR4238777}. A proof of the stability property for the Scott-Zhang interpolant is given in \cite[Corollary 3.7]{MR4238777}. To see this, we set $m=2$ and $q=2$ in estimate (3.27) of that reference to obtain 
\[
\| \calI_h w \|_{B_{2}^{\frac{3}{2}\gamma}(\Ldeux)} \lesssim \| w \|_{B_{2}^{\frac{3}{2}\gamma}(\Ldeux)} \qquad \forall w \in B_{2}^{\frac{3}{2}\gamma}(\Ldeux),
\]
where $\gamma \in (0,1)$. We then choose  $\gamma$  appropriately and use  $B_{2}^{\frac{3}{2}\gamma}(\Omega) = H^{\frac{3}{2}\gamma}(\Omega)$ to deduce, for $r \in [0,3/2)$, the bound
\[
  \| \calI_h w \|_{H^{r}(\Omega)} \lesssim \| w \|_{H^{r}(\Omega)}, \qquad \forall w \in H^{r}(\Omega).
\]

\subsection{The discrete Laplacian}
\label{sub:Deltah}

We introduce the discrete Laplacian as the linear mapping $\Delta_h : \Fespace \to \Fespace$ defined by
\begin{equation}
\label{eq:DefofDeltah}
  \int_\Omega \Delta_h v_h w_h \diff x  = - \int_\Omega \GRAD v_h \cdot \GRAD w_h \diff x \qquad \forall v_h ,w_h \in \Fespace.
\end{equation}
From its definition, it follows that the map $\Delta_h$ is symmetric, invertible, and negative definite. Therefore, the classical spectral theorem of linear algebra guarantees that there exists $\{(\Phi_n,\Lambda_n)\}_{n=1}^{N_h} \subset \Fespace \times \Real^+$ such that
\begin{equation}
\label{eq:DefofEigenPairsH}
  -\Delta_h \Phi_n = \Lambda_n \Phi_n, \qquad n = 1, \ldots, N_h,
\end{equation}
\ie
\[
  \int_{\Omega} -\Delta_h \Phi_n w_h \diff{x} = \int_{\Omega} \nabla \Phi_n \cdot \nabla w_h \diff{x} = \Lambda_n \int_{\Omega} \Phi_n w_h \diff{x} 
  \qquad
  \forall w_h \in \Fespace.
\]
The set of eigenfunctions $\{ \Phi_n \}_{n=1}^{N_h}$ is an orthonormal basis of $\Fespace$ in $\Ldeux$ and an orthogonal basis of $\Fespace$ in $\Hunz$. In other words, this family satisfies
\begin{equation}
\label{eq:orthogonal_ortonormal}
  \int_\Omega \Phi_n \Phi_m \diff x = \delta_{n,m}, 
  \qquad
  \int_\Omega \GRAD \Phi_n \cdot \GRAD  \Phi_m \diff{x} = \Lambda_n \delta_{n,m},
  \qquad
  n,m = 1, \ldots, N_h.
\end{equation}
Finally, we note that the eigenvalues satisfy the bounds
\begin{equation}
\label{eq:eigenvalues}
  C_P^{-2} \leq \Lambda_1 < \Lambda_2 \leq \cdots \leq  \Lambda_{N_h} \lesssim h^{-2},
\end{equation}
where $C_P$ denotes the best constant in Poincar\'e's inequality. The bound $C_P^{-{2}} \leq \Lambda_1$ follows from the properties of the Rayleigh quotient and a Poincar\'e inequality, and a proof of the bound $\Lambda_{N_h} \lesssim h^{-2}$ can be found in \cite[page 53]{MR2249024}; the latter relies on the quasiuniformity of the meshes.

\begin{remark}[coefficient vectors]
  Notice that $w_h \in \Fespace$ has two ``canonical'' representations. Indeed, we may write
  \[
    w_h = \sum_{n=1}^{N_h} \hat{W}_n \Phi_n,
    \qquad
    w_h = \sum_{n=1}^{N_h} W_n \phi_n.
  \]
  It is clear that, in general, $W_n \neq \hat{W}_n$, where $n \in \{1, \ldots, N_h\}$.
\end{remark}

For $r \in \Real$, we may then define $(-\Delta_h)^r : \Fespace \to \Fespace$ as follows. If
\[
  w_h = \sum_{n=1}^{N_h} \hat{W}_n \Phi_n \in \Fespace,
\]
then
\begin{equation}
\label{eq:Lapsh}
  (-\Delta_h)^r w_h \coloneqq \sum_{n=1}^{N_h} \Lambda_n^r \hat{W}_n \Phi_n \in \Fespace.
\end{equation}

How the powers of the discrete Dirichlet Laplacian approximate the powers of the Dirichlet Laplacian was studied in \cite{MR1255054}. In particular, since $\Omega$ is convex, we have that, for every $r \in [0,1]$ and $F \in L^2(\Omega)$,
\begin{equation}
\label{eq:Matsuki}
  \left\| (-\Delta)^{-r} F - (-\Delta_h)^{-r} P_h F \right\|_\Ldeux \lesssim h^{2r} \| F \|_\Ldeux.
\end{equation}
Here, $P_h$ denotes the $L^2$-projection onto $\Fespace$. Estimate \cref{eq:Matsuki} follows from an application of \cite[Theorem 1]{MR1255054}.  This result states that, in the notation of that paper, if condition $(A_{\epsilon,0})$ holds, namely,
\[
  \left\| (-\Delta)^{-1} F - (-\Delta_h)^{-1} P_h F \right\|_\Ldeux \lesssim h^{2} \| F \|_\Ldeux \qquad \forall F \in L^2(\Omega),
\]
then estimate \cref{eq:Matsuki} is valid. Condition $(A_{\epsilon,0})$ follows from a basic duality argument.

Finally, given $r \in \Real$, we denote by $\polH_h^r(\Omega)$ the space $\Fespace$ equipped with the following norm:
\begin{equation}
\label{eq:discrete_norm}
  \| w_h \|_{\polH_h^r(\Omega)} \coloneqq \left( \sum_{n=1}^{N_h} \Lambda_n^r |\hat{W}_n|^2 \right)^{\frac{1}{2}}.
\end{equation}
The use of this notation for $\polH_h^r(\Omega)$ and $ \| \cdot \|_{\polH_h^r(\Omega)}$ is motivated by the following fact.

\begin{proposition}[norm equivalence]
\label{prop:Norms}
  Let $r \in (-\tfrac12, \frac32)$. Then,
  \[
    \| w_h \|_{\polH_h^r(\Omega)} \eqsim \| w_h \|_{\polH^r(\Omega)} \qquad \forall w_h \in \Fespace.
  \]
  The implicit constants in this equivalence are independent of $h$.
\end{proposition}
\begin{proof}
  Directly from the definition of the norm $\| \cdot \|_{\polH_h^r(\Omega)}$ and the relations in \cref{eq:orthogonal_ortonormal}, it follows that
  \[
    \| w_h \|_{\polH_h^0(\Omega)} = \| w_h \|_{\Ldeux}, \qquad \| w_h \|_{\polH_h^1(\Omega)} = \| \GRAD w_h \|_{\Ldeuxd}.
  \]
  For $r \in [0,1]$, the result follows by interpolation. For the rest of the proof, we refer the reader to \cite[Lemma 2.2]{MR2461254}.
\end{proof}

\section{An ideal discrete problem}
\label{sec:DiscreteProblemVerA}

We now introduce a numerical scheme that directly follows the theory developed in \cref{sec:Analysis}. Given $s \in (\tfrac12,1)$ and $\theta \in (1-s,s)$, we define the bilinear form
\begin{equation}
\label{eq:A_h}
  \calA_h : \polH_h^{s-\theta}(\Omega) \times \polH_h^{s+\theta}(\Omega) \to \Real, 
  \qquad
  (v_h,w_h) \mapsto \calA_h (v_h,w_h) \coloneqq  \sum_{n=1}^{N_h} \Lambda_n^s \hat{V}_n \hat{W}_n.
\end{equation}
Using the definition of $\Lapsh$, given in \cref{eq:Lapsh}, and the properties of $\{ \Phi_n \}_{n=1}^{N_h}$ stated in \cref{eq:orthogonal_ortonormal}, we obtain, for every $v_h,w_h$ in $\Fespace$,
\begin{equation}
\label{eq:identity}
\begin{aligned}
  \calA_h(v_h,w_h) &= \sum_{n,m=1}^{N_h} \int_\Omega \Lapsh \Phi_n \Phi_m \hat{V}_n \hat{W}_m \diff x
    \\
    &= \int_\Omega \sum_{n=1}^{N_h} \hat{V}_n \Lapsh\Phi_n \sum_{m=1}^{N_h} \hat{W}_m \Phi_m \diff x
    = \int_\Omega \Lapsh v_h w_h \diff x.
\end{aligned}
\end{equation}

We now present our ideal discrete problem: Find $u_h \in \polH^{s-\theta}_h(\Omega)$ such that
\begin{equation}
\label{eq:IdealDiscrProblem}
  \calA_h(u_h,v_h) = \langle \mu, v_h \rangle, \qquad \forall v_h \in \polH^{s+\theta}_h(\Omega).
\end{equation}
We immediately observe that, since $\Fespace \hookrightarrow W^{1,\infty}_0(\Omega)$, the right-hand side of the previous expression is well-defined. In addition note that, if we set, in \cref{eq:IdealDiscrProblem}, $v_h = \phi_n$ with $n \in \{ 1,\ldots, N_h\}$ we obtain that
\[
  \calA_h(u_h,\phi_n) = \langle \mu, \phi_n \rangle.
\]
We now use \cref{eq:dual_basis} and write $\langle \mu, \phi_n \rangle$ as
\[
 \langle \mu, \phi_n \rangle = \sum_{m=1}^{N_h} \langle \mu, \phi_m \rangle \int_{\Omega} \phi^\star_m \phi_n \diff{x},
\]
Thus, using the identity \cref{eq:identity} we can obtain that
\[
 \int_{\Omega} \left[ \Lapsh u_h  - \sum_{m=1}^{N_h} \langle \mu, \phi_m \rangle  \phi^\star_m \right] \phi_n \diff x = 0,
 \qquad 
 n = 1, \ldots N_h.
\]
Since $\{ \phi_n \}_{n=1}^{N_h}$ is a basis for $\Fespace$, problem \cref{eq:IdealDiscrProblem} can be equivalently rewritten as
\[
  \Lapsh u_h = M_h,
  \qquad
  M_h \coloneqq \sum_{m=1}^{N_h} \langle \mu, \phi_m \rangle  \phi^\star_m,
  \qquad
  M_h \in \Fespace.
\]
Existence and uniqueness then follow immediately. In addition, we may write
\[
  u_h = (-\Delta_h)^{-s} M_h = \frac{2\sin(\pi s)}\pi \int_0^\infty t^{1-2s}(t^2 I - \Delta_h)^{-1} M_h \diff t,
\]
where we used the Balakrishnan formula \cref{eq:DefOfBalakrishnan} with $\Delta$ replaced by $\Delta_h$.

Notice, however, that none of the considerations given above yield statements that are uniform for $h>0$. The next result guarantees uniformity.

\begin{theorem}[BNBh]
\label{thm:BNBh}
Let $s \in (\tfrac12,1)$ and choose $\theta \in (1-s,s)$. The bilinear form $\calA_h$ satisfies the following inf-sup conditions:
  \begin{align}
  \label{eq:inf_sup_Ah_1}
    \inf_{v_h \in \polH_h^{s-\theta}(\Omega)}\sup_{w_h \in \polH_h^{s+\theta}(\Omega)} \frac{ \calA_h(v_h,w_h) }{ \| v_h \|_{\polH_h^{s-\theta}(\Omega)} \| w_h \|_{\polH_h^{s+\theta}(\Omega)} } & \geq 1,
    \\
    \label{eq:inf_sup_Ah_2}
    \inf_{w_h \in \polH_h^{s+\theta}(\Omega)}\sup_{v_h \in \polH_h^{s-\theta}(\Omega)} \frac{ \calA_h(v_h,w_h) }{ \| v_h \|_{\polH_h^{s-\theta}(\Omega)} \| w_h \|_{\polH_h^{s+\theta}(\Omega)} } & \geq 1.
  \end{align}
  Consequently, if $s \in (\tfrac12,\tfrac34)$, the unique solution $u_h \in \Fespace$ of problem \cref{eq:IdealDiscrProblem} satisfies
  \begin{equation}
  \label{eq:discrete_stability}
  \| u_h \|_{\polH^{s-\theta}(\Omega)} \eqsim \| u_h \|_{\polH^{s-\theta}_h(\Omega)} \lesssim \| \mu \|_{\calM(\Omega)},
  \end{equation}
  where the implicit constant is independent of $h$.
\end{theorem}
\begin{proof}
  The proof of the discrete inf-sup conditions \cref{eq:inf_sup_Ah_1} and \cref{eq:inf_sup_Ah_2} follows the approach of \cref{thm:BNB}, using the definitions of the norm $\| \cdot \|_{\mathbb{H}_h^{r}(\Omega)}$ given in \cref{eq:discrete_norm} and the bilinear form $\mathcal{A}_h$ given in \cref{eq:A_h}. The discrete stability estimate \cref{eq:discrete_stability} is derived using \cref{eq:inf_sup_Ah_1} and the equivalence stated in \cref{prop:Norms}. Note that, since $s - \theta \in (0,\tfrac12)$ and $s+\theta \in (1,\tfrac32)$, the equivalence result of \cref{prop:Norms} applies.
  For brevity, we omit the details.
\end{proof}

\subsection{Convergence}
\label{sub:convergence}

We now show the convergence of scheme \cref{eq:IdealDiscrProblem}. The argument is essentially a Strang-type result \cite[Section 27.4]{MR4269305}.

\begin{theorem}[convergence]
  Let $s \in (\tfrac12,\tfrac34)$ and $\vartheta \in (1-s,s)$. Let $u$ and $u_h $ solve problems \cref{eq:TheEqn} and \cref{eq:IdealDiscrProblem}, respectively. Then, for every $\theta \in (\vartheta,s)$, we have
  \[
    \| u - u_h \|_{\polH^{s-\theta}(\Omega)} \to 0,
    \qquad
    h \to 0.
  \]
\end{theorem}
\begin{proof}
  Let $w_h \in \Fespace$ be arbitrary. Using the equivalence from \cref{prop:Norms} and an inf-sup condition from \cref{thm:BNBh}, we deduce that
  \begin{equation}
  \label{eq:estimate_wh_uh}
    \| w_h - u_h \|_{\polH^{s-\theta}(\Omega)} 
    \eqsim 
    \| w_h - u_h \|_{\polH^{s-\theta}_h(\Omega)} 
    \leq
    \sup_{v_h \in \polH^{s+\theta}_h(\Omega)} \frac{ \calA_h(w_h - u_h, v_h) }{ \| v_h \|_{\polH^{s+\theta}_h(\Omega)} }.
  \end{equation}
  Note that $s - \theta \in (0,\frac12)$, so \cref{prop:Norms} applies. To estimate the right-hand side of the previous expression, we write the numerator $\calA_h(w_h - u_h, v_h)$ as follows:
  \begin{align*}
    \calA_h(w_h - u_h, v_h) &= \calA_h(w_h,v_h) - \calA_h(u_h,v_h) \pm \calA(w_h,v_h) \pm \calA(u,v_h) 
    \\
    &= \left( \calA_h - \calA \right)(w_h,v_h) - \langle \mu, v_h \rangle + \langle \mu, v_h \rangle + \calA(w_h - u, v_h)
    \\
    &= \left( \calA_h - \calA \right)(w_h,v_h) + \calA(w_h - u, v_h),
  \end{align*}
  where we have used that $u$ and $u_h $ solve \cref{eq:TheEqn} and \cref{eq:IdealDiscrProblem}, respectively. Substituting this identity into   \cref{eq:estimate_wh_uh}, we obtain
  \[
    \| w_h - u_h \|_{\polH^{s-\theta}(\Omega)} \lesssim
    \sup_{v_h \in \polH^{s+\theta}_h(\Omega)} \frac{ \left( \calA_h - \calA \right)(w_h,v_h) }{ \| v_h \|_{\polH^{s+\theta}_h(\Omega)} }
    +
    \sup_{v_h \in \polH^{s+\theta}_h(\Omega)} \frac{ \calA(w_h - u, v_h) }{ \| v_h \|_{\polH^{s+\theta}_h(\Omega)} }
    \eqqcolon \mathrm{I} + \mathrm{II}.
  \]
  To bound the term $\mathrm{I}$, we first note that
  \begin{align*}
    \left( \calA - \calA_h \right)(w_h,v_h) &= \tensor[_{-s-\theta}]{ \left\langle \left[ \Laps - \Lapsh \right] w_h, v_h \right\rangle}{_{s+\theta}}
    \\
    &= \tensor[_{-s-\theta}]{ \left\langle \Laps\left[ (-\Delta_h)^{-s} - (-\Delta)^{-s} \right] \Lapsh w_h, v_h \right\rangle}{_{s+\theta}}
    \\
    &= \tensor[_{s-\theta}]{ \left\langle \left[ (-\Delta_h)^{-s} - (-\Delta)^{-s} \right] \Lapsh w_h, \Laps v_h \right\rangle}{_{-s+\theta}}
    \\
    &\leq \left\| \left[ (-\Delta_h)^{-s} - (-\Delta)^{-s} \right] \Lapsh w_h \right\|_\Ldeux \| \Laps v_h \|_\Ldeux.
  \end{align*}
  We note that every $v_h \in \Fespace$ belongs to $\mathbb{H}^r(\Omega)$ for every $r<\tfrac{3}{2}$. As a result, $(-\Delta)^s v_h \in \mathbb{H}^{r-2s}(\Omega)$. Since, by assumption, $s<\tfrac34$, every $v_h \in \Fespace$ thus satisfies that $\Laps v_h \in \Ldeux$. We may now invoke an inverse inequality to obtain
  \[
    \| \Laps v_h \|_\Ldeux = \| v_h \|_{\polH^{2s}(\Omega)} \lesssim h^{-s+\theta} \| v_h \|_{\polH^{s+\theta}(\Omega)} \lesssim h^{-s+\theta} \| v_h \|_{\polH^{s+\theta}_h(\Omega)}.
  \]
  Note that $2s, s+\theta \in (1,\tfrac{3}{2})$. In addition, \cref{eq:Matsuki} with $r=s$ gives that
  \begin{align*}
    \left\| \left[ (-\Delta_h)^{-s} - (-\Delta)^{-s} \right] \Lapsh w_h \right\|_\Ldeux &\lesssim h^{2s} \| \Lapsh w_h \|_\Ldeux
    \\
    &= h^{2s} \| w_h \|_{\polH_h^{2s}(\Omega)} \lesssim h^{2s} \| w_h \|_{\polH^{2s}(\Omega)},
  \end{align*}
  where, in the last step, we used that $s < \tfrac34$ and invoked \cref{prop:Norms}. We recall that $2s \in (1,\tfrac{3}{2})$. Gathering all the previous estimates, we see that
  \[
    \mathrm{I} \lesssim h^{2s} h^{-s+\theta}  \| w_h \|_{\polH^{2s}(\Omega)} =  h^{s+\theta} \| w_h \|_{\polH^{2s}(\Omega)} \lesssim h^{\theta-\vartheta} \| w_h \|_{\polH^{s-\vartheta}(\Omega)}.
  \]
  In the last step, we invoked an inverse inequality. Note that $\theta - \vartheta >0$.
  
  The estimate for $\mathrm{II}$ is fairly straightforward. In fact, we have
  \[
    \mathrm{II} \lesssim \| w_h - u \|_{\polH^{s-\theta}(\Omega)}.
  \]

  We can now proceed to show convergence. Fix $\vare>0$. It is known that there exists $u_\vare \in C_0^\infty(\Omega)$ such that
  \[
    \| u - u_\vare \|_{\polH^{s-\theta}(\Omega)} \lesssim \| u - u_\vare \|_{\polH^{s-\vartheta}(\Omega)} < \vare.
  \]
  Note that we have used that $\vartheta < \theta$. We now estimate $\| u - u_h \|_{\polH^{s-\theta}(\Omega)}$ as follows:
  \begin{align*}
    \| u - u_h \|_{\polH^{s-\theta}(\Omega)} &\leq \| u - u_\vare \|_{\polH^{s-\theta}(\Omega)} + \| u_\vare - w_h \|_{\polH^{s-\theta}(\Omega)} + \| w_h - u_h \|_{\polH^{s-\theta}(\Omega)}
    \\
    &\leq C_1 \vare + \| u_\vare - w_h \|_{\polH^{s-\theta}(\Omega)} + \| w_h - u_h \|_{\polH^{s-\theta}(\Omega)}
    \\
    &\leq C_2 \left(\vare + \| u_\vare - w_h \|_{\polH^{s-\theta}(\Omega)} + h^{\theta-\vartheta} \| w_h \|_{\polH^{s-\vartheta}(\Omega)} + \| w_h - u \|_{\polH^{s-\theta}(\Omega)} \right)
    \\
    &\leq C_2 \left( 2 \vare + 2\| u_\vare - w_h \|_{\polH^{s-\theta}(\Omega)}+ h^{\theta-\vartheta} \| w_h \|_{\polH^{s-\vartheta}(\Omega)} \right),
  \end{align*}
  where the constants $C_1$ and $C_2$ are independent of $h$ and $\vare$. Having chosen $u_\varepsilon$, we select $w_h \in \Fespace$ accordingly. Set $w_h = \calI_h u_\vare$ and choose $r = s-\vartheta$, $t = s-\theta$, and $m=s-\vartheta$ in \cref{eq:Interpolant_1} and \cref{eq:Interpolant_2} to obtain
  \begin{align*}
    \| \calI_h u_\vare \|_{\polH^{s-\vartheta}(\Omega)} &\leq C_3 \|  u_\vare \|_{\polH^{s-\vartheta}(\Omega)},
    \\
    \| u_\vare - \calI_h u_\vare \|_{\polH^{s-\theta}(\Omega)} &\leq C_3 h^{\theta - \vartheta} \| u_\vare \|_{\polH^{s-\vartheta}(\Omega)},
  \end{align*}
  where $C_3$ is independent of $w$ and $h$. We note that $r = s - \vartheta \in (0,\tfrac12)$, $t = s - \theta \in (0,\tfrac12)$, and $m = r$. In summary,
  \[
    \| u - u_h \|_{\polH^{s-\theta}(\Omega)} \leq 2C_2\vare + C_2(2C_3 + 1) h^{\theta-\vartheta} \| u_\vare \|_{\polH^{s-\vartheta}(\Omega)}.
  \]
  We can now set $h\ll1$, sufficiently small, and convergence follows.
\end{proof}

\subsection{Rates of convergence}
\label{sub:Rates}

In the previous section, provided $s \in(\tfrac12,\tfrac34)$, we obtained convergence without rates for scheme \cref{eq:IdealDiscrProblem}. This is not surprising, as the regularity of the solution is barely sufficient for the problem to be well-defined. Therefore, if rates of convergence are desired, the error must be measured in a weaker norm. We do that here.

\begin{theorem}[convergence rate]
\label{thm:Duality}
Let $s \in (\tfrac12,\tfrac34)$ and $\theta \in (1-s,s)$. If $u \in \polH^{s-\theta}(\Omega)$ and $u_h \in \polH^{s-\theta}_h(\Omega)$ denote the solutions to \cref{eq:TheEqn} and \cref{eq:IdealDiscrProblem}, respectively, then
  \begin{equation}
  \label{eq:error bound_first}
    \| u - u_h \|_\Ldeux \lesssim h^{s-\theta} \| \mu \|_{\calM(\Omega)}.
  \end{equation}
\end{theorem}
\begin{proof}
  The proof proceeds similarly to the classical Aubin-Nitsche duality result but includes an additional step to account for the consistency error, which is the difference between $\calA$ and $\calA_h$.
  
  Define $e \coloneqq u-u_h \in \polH^{s-\theta}(\Omega)$ and consider the dual problem: Find $\zeta \in\polH^{s+\theta}(\Omega)$ such that
  \begin{equation}
    \calA(w,\zeta) = \int_\Omega e w \diff x \quad \forall w \in \polH^{s-\theta}(\Omega).
    \label{eq:dual_problem}
  \end{equation}
  Owing to \cref{thm:BNB}, this problem has a unique solution. On the other hand, it is clear that this problem implies 
  \[
    \sum_{k=1}^\infty \lambda_k^s w_k \zeta_k = \sum_{k=1}^\infty e_k w_k \qquad \implies \qquad \zeta_k = \lambda_k^{-s} e_k, \quad \forall k \in \mathbb{N}.
  \]
  Therefore,
  \[
    \| \zeta \|_{\polH^{2s}(\Omega)}^2 = \sum_{k=1}^\infty \lambda_k^{2s} \left| \lambda_k^{-s} e_k \right|^2 = \sum_{k=1}^\infty |e_k|^2 = \| u- u_h \|_\Ldeux^2.
  \]
  We now let $\zeta_h \in \Fespace$ but otherwise arbitrary, and use the fact that the error $e$ is an admissible test function in the dual problem \cref{eq:dual_problem} to obtain
  \begin{align*}
    \| e \|_\Ldeux^2 &= \calA(e,\zeta) = \langle \mu, \zeta \rangle - \calA(u_h,\zeta) 
%     \pm \langle \mu, \zeta_h \rangle 
    \pm \calA_h(u_h,\zeta_h) 
    \\
    &= \langle \mu, \zeta - \zeta_h \rangle + \left( \calA_h -\calA\right)(u_h, \zeta_h) + \calA(u_h, \zeta_h - \zeta)
    = \mathrm{I} + \mathrm{II} + \mathrm{III}.
  \end{align*}
  We consider each term individually.

  To bound the term $\mathrm{I}$, we first observe that $2s \in (1,\tfrac32)$ and $s + \theta \in (1,\tfrac32)$. Next, we choose $\zeta_h = \calI_h \zeta$ and set $t=s+\theta$ and $m = 2s$ in \cref{eq:Interpolant_2} to obtain
  \[
    \mathrm{I} \lesssim \| \mu \|_{\polH^{-s-\theta}(\Omega)} \| \zeta - \zeta_h \|_{\polH^{s+\theta}(\Omega)} \lesssim h^{s-\theta} \| \mu \|_{\calM(\Omega)} \| \zeta \|_{\polH^{2s}(\Omega)} = h^{s-\theta} \| \mu \|_{\calM(\Omega)} \| e \|_\Ldeux.
  \]
  Having chosen $\zeta_h$, we use the continuity of $\calA$ to easily bound the third term:
  \begin{align*}
    \mathrm{III} &\leq \| u_h \|_{\polH^{s-\theta}(\Omega)} \| \zeta - \zeta_h \|_{\polH^{s+\theta}(\Omega)}
      \lesssim h^{s-\theta} \| u_h \|_{\polH^{s-\theta}_h(\Omega)} \| \zeta \|_{\polH^{2s}(\Omega)}
      \\
      &\lesssim h^{s-\theta} \| \mu \|_{\calM(\Omega)} \| e \|_\Ldeux,
  \end{align*}
  where we used the discrete stability estimate \cref{eq:discrete_stability} and the fact that $\| \zeta \|_{\polH^{2s}(\Omega)} = \| e \|_\Ldeux$. Finally, we estimate the consistency error, which is encoded in the term $\mathrm{II}$:
  \begin{align*}
     \left( \calA_h - \calA \right)(u_h , \zeta_h)  &= \tensor[_{-s-\theta}]{ \left\langle  \left[ \Lapsh - \Laps \right]u_h, \zeta_h \right\rangle }{_{s+\theta}}
    \\
    &= \tensor[_{-s-\theta}]{ \left\langle  \Laps \left[ (-\Delta)^{-s} - (-\Delta_h)^{-s} \right] \Lapsh u_h, \zeta_h \right\rangle }{_{s+\theta}}
    \\
    &= \tensor[_{s-\theta}]{ \left\langle \left[ (-\Delta)^{-s} - (-\Delta_h)^{-s} \right] \Lapsh u_h, \Laps \zeta_h \right\rangle}{_{-s+\theta}}
    \\
    &\leq \left\| \left[ (-\Delta)^{-s} - (-\Delta_h)^{-s} \right] \Lapsh u_h \right\|_\Ldeux \| \Laps \zeta_h \|_\Ldeux
    \\
    &\lesssim h^{2s} \| \Lapsh u_h \|_\Ldeux \|\zeta_h \|_{\polH^{2s}(\Omega)},
  \end{align*}
  where we have used \cref{eq:Matsuki}. We may now proceed as follows:
  \begin{align*}
    \left| \left( \calA_h - \calA \right)( u_h , \zeta_h) \right| &\lesssim h^{2s} \| u_h \|_{\polH^{2s}_h(\Omega)} \| \zeta_h \|_{\polH^{2s}(\Omega)} \lesssim h^{2s} h^{-\theta-s} \| u_h \|_{\polH^{s-\theta}_h(\Omega)} \| e \|_\Ldeux
    \\
    &\lesssim h^{s-\theta} \| \mu\|_{\calM(\Omega)} \| e \|_\Ldeux,
  \end{align*}
  where we used the inverse estimate
  \[
    \| u_h \|_{\polH_h^{2s}(\Omega)} \lesssim h^{-s-\theta} \| u_h \|_{\polH_h^{s-\theta}(\Omega)},
  \]
  which follows from the definition of $\| \cdot \|_{\polH^{r}_h(\Omega)}$ given in \cref{eq:discrete_norm} and the properties satisfied by the discrete eigenvalues stated in \cref{eq:eigenvalues}, the stability bound \cref{eq:Interpolant_1} with $r = 2s$, the fact that $\| \zeta \|_{\polH^{2s}(\Omega)} = \| e \|_\Ldeux$, and the discrete stability bound \cref{eq:discrete_stability}.

  In summary, we have deduced the following:
  \[
    \| e \|_\Ldeux^2 = \mathrm{I} + \mathrm{II} + \mathrm{III} \lesssim h^{s-\theta} \| \mu \|_{\calM(\Omega)} \| e \|_\Ldeux.
  \]
  This concludes the proof.
\end{proof}

\section{A practical scheme}
\label{sec:Gordito}

Scheme \cref{eq:IdealDiscrProblem} converges optimally in terms of regularity in $L^2(\Omega)$; see \cref{thm:Duality}. However, it is not practical. It requires knowledge of the spectral decomposition of the discrete Laplacian. In addition, its analysis requires us to restrict the range of $s$, \ie $s < \tfrac34$.

To address these shortcomings we now present a practical scheme that remains convergent in $L^2(\Omega)$ with the optimal rate $\mathcal{O}(h^{s-\theta})$. The idea is to use the scheme proposed in \cite{Sawyer} for a suitably regularized right-hand side.

\subsection{The diagonalization scheme}
\label{subsec:diagonalization}
Let $F \in \Ldeux$ and assume that $\Psi = (-\Delta)^{-s} F$. To approximate $\Psi$, after introducing the mesh $\Triang$ and the corresponding finite element space $\Fespace$, we choose $\calY > 0$, $\calK \in \Natural$, and define the parameters
\begin{equation}
\label{eq:Eigens}
  \Upsilon_k \coloneqq \left(\frac{\eta_k}\calY \right)^2,
  \qquad
  \psi_k \coloneqq  \frac{4 \sin(\pi s)}{\Upsilon_k^s \calY^2 \pi J_{1-s}(\eta_k)^2 },
  \qquad k = 1, \ldots, \calK,
\end{equation}
where $J_{\nu}$ is the Bessel function of the first kind ($\nu \in (0,1)$), and $\eta_k$ is the $k$-th positive root of $J_{-s}$. Next, for $k = 1, \ldots, \calK$, we define $\Psi_k \in \Fespace$ as the solution to
\begin{equation}
\label{eq:DefofUk}
  \int_\Omega \GRAD \Psi_k \cdot \GRAD v_h \diff x + \Upsilon_k \int_\Omega \Psi_k v_h \diff x = \int_\Omega F v_h \diff x \qquad \forall v_h \in \Fespace.
\end{equation}
Finally, we combine these results to obtain an approximate solution \cite[Section 4.2]{Sawyer}
\begin{equation}
\label{eq:DefofuhYK}
  \Psi_{h,\calY}^\calK = \sum_{k=1}^\calK \psi_k \Psi_k \in \Fespace.
\end{equation}

The convergence properties of this method are summarized in the following result.

\begin{proposition}[error estimate]
\label{prop:Sawyer}
  Let $s \in (0,1)$, $F \in \Ldeux$, $\Psi = (-\Delta)^{-s} F$, and $\Psi_{h,\calY}^\calK \in \Fespace$ be defined as in \cref{eq:DefofuhYK}. Then we have
  \[
    \| \Psi - \Psi_{h,\calY}^\calK\|_\Ldeux \lesssim \left[ h^{2s} + \exp \left(  -\frac{\calY}{\sqrt{C_P}}\right) + \left( \frac\calY\calK \right)^{2s} \right] \| F \|_\Ldeux,
  \]
  where $C_P$ is the best constant in Poincar\'e's inequality. In particular, we may choose $\calY \eqsim 2s|\log h|$ and $\calK \eqsim \tfrac\calY{h}$ to obtain
  \[
    \| \Psi - \Psi_{h,\calY}^\calK\|_\Ldeux \lesssim h^{2s} \| F \|_\Ldeux.
  \]
\end{proposition}
\begin{proof}
  See \cite[Corollary 4.12]{Sawyer}.
\end{proof}

We refer to \cite{Sawyer} for motivation and additional properties of this scheme.

\subsection{A practical scheme}

We are now ready to describe our scheme. We assume that, for $\vare  >0$, we have a regularization $\mu_\vare \in \Ldeux$ of $\mu \in \mathcal{M}(\Omega)$ and that this regularization satisfies the following estimates:
\begin{align}
\label{eq:NegNormErrEst}
  \| \mu - \mu_\vare \|_{\polH^{-s-\theta}(\Omega) } & \lesssim \vare^{s+\theta-1},
  \\
% \end{equation}
% and
% \begin{equation}
\label{eq:BlowUpL2Norm}
  \| \mu_\vare \|_\Ldeux & \lesssim \vare^{-1}.
\end{align}
In \cref{pro:convolution} below, we provide an example of such a regularization. As the next step, following the diagonalization scheme described in \cref{subsec:diagonalization}, we choose $\calY>0$ and $\calK \in \Natural$, and define $\{(\Upsilon_k,\psi_k)\}_{k=1}^\calK$ as in \cref{eq:Eigens}. For $k=1,\ldots,\calK$, we then introduce the function $U_k^\vare \in \Fespace$ as the solution to \cref{eq:DefofUk}, but with the right-hand side $F$ replaced by the regularization $\mu_\vare$, \ie $U_k^\vare \in \Fespace$ solves
\begin{equation}
\label{eq:DefOfUkvare}
  \int_\Omega \GRAD U_k^\vare \cdot \GRAD v_h \diff x + \Upsilon_k \int_\Omega U_k^\vare v_h \diff x = \int_\Omega \mu_\vare v_h \diff x \qquad \forall v_h \in \Fespace.
\end{equation}
We then define the discrete solution as
\begin{equation}
\label{eq:DefOfuhYKvare}
  u_{h,\calY}^{\calK,\vare} = \sum_{k=1}^\calK \psi_k U_k^\vare.
\end{equation}

The main convergence properties of this scheme are as follows.

\begin{theorem}[error estimate]
\label{thm:MainEst}
  Let $s \in ( \tfrac12,1)$, $u$ be the solution to \cref{eq:TheEqn}, and let $u_{h,\calY}^{\calK,\vare} \in \Fespace$ be defined as in \cref{eq:DefOfuhYKvare} with
  \[
    \calY \eqsim 2s|\log h|, \qquad \calK \eqsim \frac\calY{h}, \qquad \vare \eqsim h.
  \]
  Then, for $h \leq 1$, we have the following a priori error bound:
  \begin{equation}
  \label{eq:error bound_second}
    \| u - u_{h,\calY}^{\calK,\vare} \|_\Ldeux \lesssim h^{s+\theta-1}.
  \end{equation}
\end{theorem}
\begin{proof}
  Define $u_\vare \coloneqq (-\Delta)^{-s} \mu_\vare \in \polH^s(\Omega)$. Since $\mu_{\vare} \in L^2(\Omega)$, we additionally have $u_\vare \in \polH^{2s}(\Omega)$. Using the stability estimate \cref{eq:continuous_stability} we deduce
  \[
    \| u - u_\vare \|_\Ldeux \lesssim \| u - u_\vare \|_{\polH^{s-\theta}(\Omega)} \lesssim \| \mu - \mu_\vare \|_{\polH^{-s-\theta}(\Omega)} \lesssim \vare^{s+\theta-1},
  \]
  where the last estimate follows from \cref{eq:NegNormErrEst}. Next, we apply the error bound from \cref{prop:Sawyer} with $F = \mu_\vare$ to conclude that
  \[
    \| u_\vare - u_{h,\calY}^{\calK,\vare} \|_\Ldeux \lesssim h^{2s} \| \mu_\vare \|_\Ldeux \lesssim h^{2s-1} \leq h^{s+\theta-1},
  \]
  where we used \cref{eq:BlowUpL2Norm}, the scaling for $\vare$, specifically $\vare \eqsim h$, the inequality $\theta<s$, and the fact that $h \leq 1$. We conclude using the triangle inequality.
\end{proof}

\begin{remark}[bound \cref{eq:error bound_first} vs.~bound \cref{eq:error bound_second}]
  Notice that for the scheme defined in \cref{eq:IdealDiscrProblem}, we proved a convergence rate of order $\calO(h^{s-\theta})$, whereas our practical scheme \cref{eq:DefOfuhYKvare} has a convergence rate of order $\calO(h^{s+\theta-1})$. 
  To compare these assume that $s \in (\tfrac12,\tfrac34)$. Using that $\theta \in (1-s,s)$ we get
  \[
    0 < s+\theta-1 < 2s -1, 
    \qquad 
    0< s-\theta < 2s-1. 
%     \qquad
%     2s - 1 \in (0, \tfrac12).
  \]
  In other words, an appropriate choice of $\theta$ gives the rate of convergence of order $\calO(h^{2s-1-\delta})$ for both schemes, where $\delta > 0$ is arbitrarily small. It is important to note, in addition, that scheme \cref{eq:DefOfuhYKvare} converges with this rate for $s \in(\tfrac12,1)$.
\end{remark}

\subsection{Suitable regularizations}

Recall that the practical scheme \cref{eq:DefOfuhYKvare} relies on the possibility of constructing suitable regularizations $\mu_{\vare}$ of $\mu$. In the particular case of $\mu = \delta_{\vertex}$ with $\vertex \in \Omega$, such constructions can be found, for instance, in \cite{MR337032,MR3429589,MR4169484}. Here, we show that regularization by convolution has all the required properties.

\begin{proposition}[convolution]
\label{pro:convolution}
  Let $\mu \in \calM(\Omega)$, and let $\mu_\vare \in C_0^\infty(\Omega)$ be given by convolution at scale $\vare>0$, \ie
  \[
    \mu_\vare = \mu \star \rho_\vare,
  \]
  where $\rho$ is a standard mollifier and, as usual, $\rho_\vare(x) = \tfrac1{\vare^2} \rho(x/\vare)$. Then, $\mu_\vare$ satisfies \cref{eq:NegNormErrEst} and \cref{eq:BlowUpL2Norm}.
\end{proposition}
\begin{proof}
  We first note that, following \cite[Proposition 8.49]{MR1681462}, for almost every $x \in \Real^2$, the integral
  \[
    \mu_\vare(x) \coloneqq \int \rho_\vare(x-y) \diff \mu(y)
  \]
  exists and it is finite. In addition, for every $p \in [1,\infty]$,
  \[
    \mu_\vare \in L^p(\Real^2), \qquad \supp \mu_\vare \subset \bar\Omega + B_\vare.
  \]
  We now prove \cref{eq:NegNormErrEst}. Recall that, since $1<s+\theta<2s$, we have $\polH^{s+\theta}(\Omega) \hookrightarrow C^{0,\gamma}(\bar\Omega)$, where $\gamma = s+\theta-1 \in (0,\tfrac12)$; \cite[Theorem 4.12, \textbf{Part C}]{MR2424078}. Now let $w \in \polH^{s+\theta}(\Omega)$ and estimate
  \begin{align*}
    \tensor[_{-s-\theta}]{\langle \mu- \mu_\vare, w \rangle}{_{s+\theta}} &= 
      \int_\Omega w(y) \diff \mu(y) - \int_\Omega w(x) \left[ \int \rho_\vare(x-y) \diff\mu(y) \right] \diff x
      \\
    &= \int_\Omega \left[ w(y) - \int \rho_\vare(x-y) w(x) \diff x \right] \diff \mu(y).
  \end{align*}
  Next, since $w$ is continuous, an application of the mean value theorem shows that, for each fixed $y \in \Omega$ there is $z \in B_\vare(y) \cap \Omega$ for which \cite[Corollary 11.3.4]{MR2033095}
  \[
    w(z) = \int \rho_\vare(x-y) w(x) \diff x.
  \]
  We may thus continue using the fact that $w$ is H\"older continuous and estimate
  \begin{align*}
    \left| \tensor[_{-s-\theta}]{\langle \mu- \mu_\vare, w \rangle}{_{s+\theta}} \right| &\leq 
      \int_\Omega \left| w(y) - w(z) \right| \diff |\mu|(y) 
    \lesssim \vare^\gamma |w|_{C^{0,\gamma}(\bar\Omega)} |\mu|(\Omega)¨
    \\
    &\lesssim \vare^\gamma \|w\|_{\polH^{s+\theta}(\bar\Omega)} |\mu|(\Omega).
  \end{align*}
  Finally,
  \[
    \| \mu - \mu_\vare \|_{\polH^{-s-\theta}(\Omega)} = \sup_{w \in \polH^{s+\theta}(\Omega) } \frac{ \tensor[_{-s-\theta}]{\langle \mu- \mu_\vare, w \rangle}{_{s+\theta}} }{ \|w\|_{\polH^{s+\theta}(\Omega)} } \lesssim \vare^\gamma  |\mu|(\Omega).
  \]
  
  The proof of \cref{eq:BlowUpL2Norm} is simply an application of Minkowski's integral inequality
  \begin{align*}
    \| \mu_\vare \|_\Ldeux &= \left( \int_\Omega \left| \int \rho_\vare(x-y) \diff \mu(y) \right|^2 \diff x \right)^{1/2}
      \leq \int \left( \int \rho_\vare(x-y)^2 \diff x \right)^{1/2} \diff |\mu|(y) 
    \\
    &\leq |\mu|(\Omega) \left( \int \rho_\vare(z) \diff z \right)^{1/2} \| \rho \|_{L^\infty(\Real^2)}^{1/2} \vare^{-1}.
  \end{align*}
\end{proof}

\section{Numerical Illustrations}
\label{sec:NumExp}

Let us conclude our discussion by presenting some numerical illustrations. In the absence of exact solutions to illustrate the rates of convergence our illustrations are of a qualitative nature. These were produced using the \texttt{FreeFem++} package \cite{MR3043640}.

In our experiments we set $\Omega = (0,1)^2$ and $s = 0.65$. The triangulation is uniform with $h = 3.89105\times 10^{-3}$, so that $\dim V(\Triang) = 65536$. Following \cref{thm:MainEst} we set
\[
  \calY = 11.0982, \qquad \calK = 2852.
\]

\begin{figure}
  \begin{center}
    \includegraphics[scale=0.25]{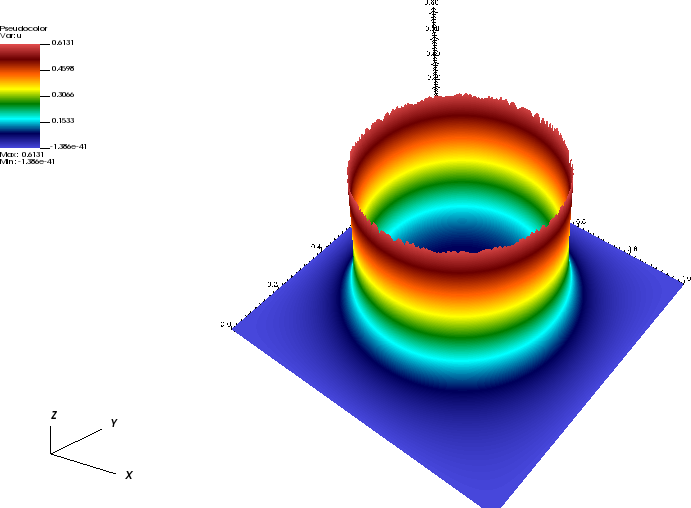}
    \includegraphics[scale=0.25]{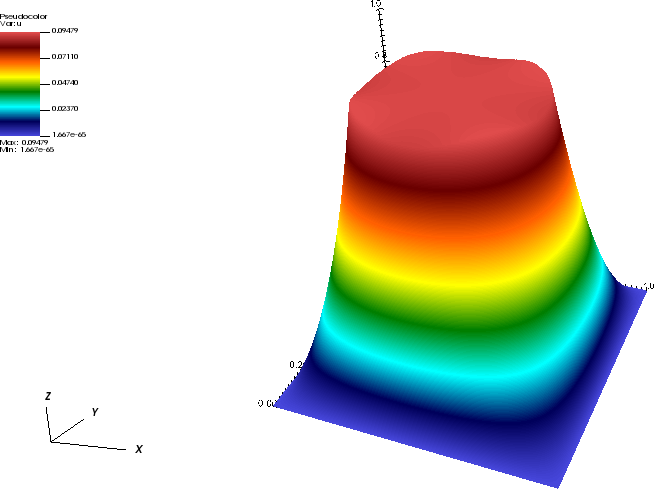}
  \end{center}
\caption{The left panel shows the solution to the fractional Laplacian ($s=0.65$) with the right hand side given by \cref{eq:Curve}. For comparison, the right panel shows the solution to the Laplacian with the same right hand side.}
\label{fig:curve}
\end{figure}

In the first experiment we let $r=0.3$ and 
\begin{equation}
\label{eq:Curve}
  \mu = \frac1{2\pi r} \delta_\Gamma, \qquad \Gamma \coloneqq \left\{ (x,y)^\top \in \Real^2 \ : \ (x-\tfrac12)^2 + (y-\tfrac12)^2 = r^2 \right\},
\end{equation}
which we regularize to
\[
  \mu_\vare = \frac1{4\pi h r} \chi_{\frakR}, \quad \frakR \coloneqq \left\{ (x,y)^\top \in \Real^2 \ : \ (r-h)^2 < (x-\tfrac12)^2 + (y-\tfrac12)^2 = (r + h)^2 \right\}.
\]
Figure~\ref{fig:curve} shows the solution to \cref{eq:TheEqn}. For comparison, this figure also shows the solution to $-\Delta w = \mu_\vare$ supplemented with homogeneous Dirichlet boundary conditions. Observe the stark difference in the support properties of the solutions.

\begin{figure}
  \begin{center}
    \includegraphics[scale=0.25]{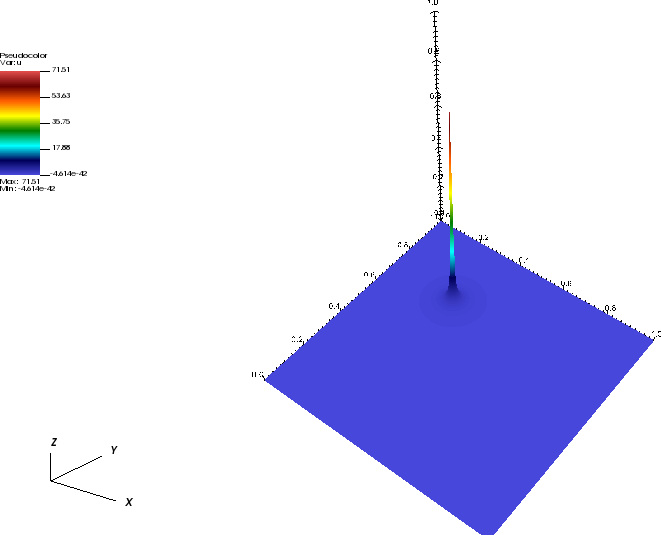}
    \includegraphics[scale=0.25]{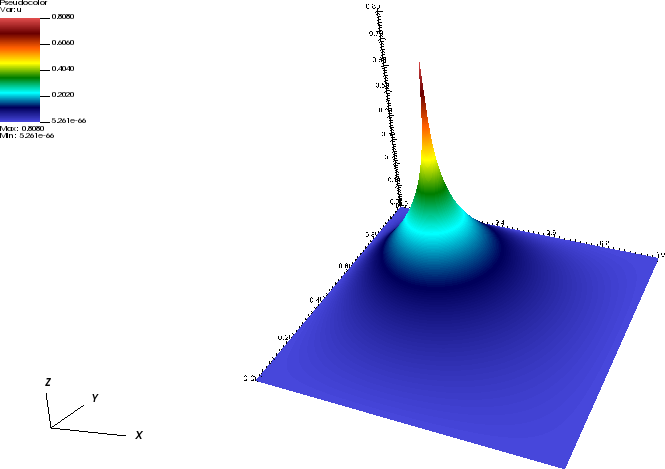}
  \end{center}
\caption{The left panel shows the solution to the fractional Laplacian ($s=0.65$) with the right hand side given by \cref{eq:Dirac}. For comparison, the right panel shows the solution to the Laplacian whit the same right hand side.}
\label{fig:Dirac}
\end{figure}

In our second experiment we consider
\begin{equation}
\label{eq:Dirac}
  \mu = \delta_{\vertex}, \qquad \vertex = (0.3,0.7),
\end{equation}
which we regularize to
\[
  \mu_\vare = \frac1{\pi h^2}\chi_\frakB, \qquad \frakB \coloneqq \left\{ (x,y)^\top \in \Real^2 \ : \ (x-\vertex_1)^2 + (y-\vertex_2)^2 < h^2 \right\}.
\]
The results of our computations are shown in Figure~\ref{fig:Dirac}. Once again, for comparison we show the solution of the Dirichlet Laplacian with the same right hand side. We note the steeper singularity for the solution of \cref{eq:TheEqn}.

\section*{Acknowledgements}

The work of EO has been partially supported by ANID grant FONDECYT-1220156 and by USM through USM project 2025 PI LII 25 12.
The work of AJS has been partially supported by NSF grant DMS-2409918.

\bibliographystyle{siamplain}
\bibliography{biblio}

@article {MR1173747,
    AUTHOR = {Oswald, P.},
     TITLE = {On the boundedness of the mapping {$f\to|f|$} in {B}esov
              spaces},
   JOURNAL = {Comment. Math. Univ. Carolin.},
  FJOURNAL = {Commentationes Mathematicae Universitatis Carolinae},
    VOLUME = {33},
      YEAR = {1992},
    NUMBER = {1},
     PAGES = {57--66},
      ISSN = {0010-2628,1213-7243},
   MRCLASS = {46E35 (41A15 47H30)},
  MRNUMBER = {1173747},
}

@book {MR2583281,
    AUTHOR = {Tr\"oltzsch, Fredi},
     TITLE = {Optimal control of partial differential equations},
    SERIES = {Graduate Studies in Mathematics},
    VOLUME = {112},
 PUBLISHER = {American Mathematical Society, Providence, RI},
      YEAR = {2010},
     PAGES = {xvi+399},
      ISBN = {978-0-8218-4904-0},
   MRCLASS = {49-02 (35J61 35K59 49J20 49K20 93C20)},
  MRNUMBER = {2583281},
MRREVIEWER = {Irena\ Lasiecka},
       DOI = {10.1090/gsm/112},
       URL = {https://doi.org/10.1090/gsm/112},
}

@book {MR2033095,
    AUTHOR = {Zorich, Vladimir A.},
     TITLE = {Mathematical analysis. {II}},
    SERIES = {Universitext},
   EDITION = {{R}ussian},
 PUBLISHER = {Springer-Verlag, Berlin},
      YEAR = {2004},
     PAGES = {xvi+681},
      ISBN = {3-540-40633-6},
   MRCLASS = {00A05 (26-01 58-01)},
  MRNUMBER = {2033095},
}

@book {MR1681462,
    AUTHOR = {Folland, Gerald B.},
     TITLE = {Real analysis},
    SERIES = {Pure and Applied Mathematics (New York)},
   EDITION = {Second},
      NOTE = {Modern techniques and their applications,
              A Wiley-Interscience Publication},
 PUBLISHER = {John Wiley \& Sons, Inc., New York},
      YEAR = {1999},
     PAGES = {xvi+386},
      ISBN = {0-471-31716-0},
   MRCLASS = {00A05 (26-01 28-01 46-01)},
  MRNUMBER = {1681462},
}

@article {MR2461254,
    AUTHOR = {Guermond, Jean-Luc and Pasciak, Joseph E.},
     TITLE = {Stability of discrete {S}tokes operators in fractional
              {S}obolev spaces},
   JOURNAL = {J. Math. Fluid Mech.},
  FJOURNAL = {Journal of Mathematical Fluid Mechanics},
    VOLUME = {10},
      YEAR = {2008},
    NUMBER = {4},
     PAGES = {588--610},
      ISSN = {1422-6928,1422-6952},
   MRCLASS = {35Q30 (35B35 65N12 65N30 76M10)},
  MRNUMBER = {2461254},
MRREVIEWER = {Jordi\ Blasco},
       DOI = {10.1007/s00021-007-0244-z},
}

@book {MR2249024,
    AUTHOR = {Thom\'ee, Vidar},
     TITLE = {Galerkin finite element methods for parabolic problems},
    SERIES = {Springer Series in Computational Mathematics},
    VOLUME = {25},
   EDITION = {Second},
 PUBLISHER = {Springer-Verlag, Berlin},
      YEAR = {2006},
     PAGES = {xii+370},
      ISBN = {978-3-540-33121-6; 3-540-33121-2},
   MRCLASS = {65-02 (65M15 65M60)},
  MRNUMBER = {2249024},
}

@article {MR3356020,
    AUTHOR = {Bonito, Andrea and Pasciak, Joseph E.},
     TITLE = {Numerical approximation of fractional powers of elliptic
              operators},
   JOURNAL = {Math. Comp.},
  FJOURNAL = {Mathematics of Computation},
    VOLUME = {84},
      YEAR = {2015},
    NUMBER = {295},
     PAGES = {2083--2110},
      ISSN = {0025-5718,1088-6842},
   MRCLASS = {65N30 (65R20)},
  MRNUMBER = {3356020},
MRREVIEWER = {Igor\ Bock},
       DOI = {10.1090/S0025-5718-2015-02937-8},
}

@article {MR3489634,
    AUTHOR = {Caffarelli, Luis A. and Stinga, Pablo Ra\'ul},
     TITLE = {Fractional elliptic equations, {C}accioppoli estimates and
              regularity},
   JOURNAL = {Ann. Inst. H. Poincar\'e{} C Anal. Non Lin\'eaire},
  FJOURNAL = {Annales de l'Institut Henri Poincar\'e{} C. Analyse Non
              Lin\'eaire},
    VOLUME = {33},
      YEAR = {2016},
    NUMBER = {3},
     PAGES = {767--807},
      ISSN = {0294-1449,1873-1430},
   MRCLASS = {35R11 (35B45 35B65 46E35)},
  MRNUMBER = {3489634},
MRREVIEWER = {Mark\ Allen},
       DOI = {10.1016/j.anihpc.2015.01.004},
}

@book {MR3753604,
    AUTHOR = {Lunardi, Alessandra},
     TITLE = {Interpolation theory},
    SERIES = {Appunti. Scuola Normale Superiore di Pisa (Nuova Serie)},
    VOLUME = {16},
   EDITION = {Third},
 PUBLISHER = {Edizioni della Normale, Pisa},
      YEAR = {2018},
     PAGES = {xiv+199},
      ISBN = {978-88-7642-639-1; 978-88-7642-638-4},
   MRCLASS = {46M35 (46-02 46B70 47D06 47F05)},
  MRNUMBER = {3753604},
       DOI = {10.1007/978-88-7642-638-4},
}

@article {MR216336,
    AUTHOR = {Fujiwara, Daisuke},
     TITLE = {Concrete characterization of the domains of fractional powers
              of some elliptic differential operators of the second order},
   JOURNAL = {Proc. Japan Acad.},
  FJOURNAL = {Proceedings of the Japan Academy},
    VOLUME = {43},
      YEAR = {1967},
     PAGES = {82--86},
      ISSN = {0021-4280},
   MRCLASS = {47.65},
  MRNUMBER = {216336},
MRREVIEWER = {K.\ Gustafson},
       URL = {http://projecteuclid.org/euclid.pja/1195521686},
}

@incollection{NV,
  author = {Nochetto, R. H. and Veeser, A.},
  title = {Primer of Adaptive Finite Element Methods},
  booktitle ={Multiscale and Adaptivity: Modeling, Numerics and
               Applications, CIME Lectures},
  publisher = {Springer},
  year = {2011},
}

@book {MR2424078,
    AUTHOR = {Adams, Robert A. and Fournier, John J. F.},
     TITLE = {Sobolev spaces},
    SERIES = {Pure and Applied Mathematics (Amsterdam)},
    VOLUME = {140},
   EDITION = {Second},
 PUBLISHER = {Elsevier/Academic Press, Amsterdam},
      YEAR = {2003},
     PAGES = {xiv+305},
      ISBN = {0-12-044143-8},
   MRCLASS = {46E35 (46-01 46-02 46B70 46Exx)},
  MRNUMBER = {2424078},
}

@article {MR2944369,
    AUTHOR = {Di Nezza, Eleonora and Palatucci, Giampiero and Valdinoci,
              Enrico},
     TITLE = {Hitchhiker's guide to the fractional {S}obolev spaces},
   JOURNAL = {Bull. Sci. Math.},
  FJOURNAL = {Bulletin des Sciences Math\'ematiques},
    VOLUME = {136},
      YEAR = {2012},
    NUMBER = {5},
     PAGES = {521--573},
      ISSN = {0007-4497,1952-4773},
   MRCLASS = {46E35 (35A23 35S05 35S30)},
  MRNUMBER = {2944369},
MRREVIEWER = {Lanzhe\ Liu},
       DOI = {10.1016/j.bulsci.2011.12.004},}

@book {MR3409135,
    AUTHOR = {Evans, Lawrence C. and Gariepy, Ronald F.},
     TITLE = {Measure theory and fine properties of functions},
    SERIES = {Textbooks in Mathematics},
   EDITION = {Revised},
 PUBLISHER = {CRC Press, Boca Raton, FL},
      YEAR = {2015},
     PAGES = {xiv+299},
      ISBN = {978-1-4822-4238-6},
   MRCLASS = {28-01},
  MRNUMBER = {3409135},
}

@article {MR3343061,
    AUTHOR = {Chandler-Wilde, S. N. and Hewett, D. P. and Moiola, A.},
     TITLE = {Interpolation of {H}ilbert and {S}obolev spaces: quantitative
              estimates and counterexamples},
   JOURNAL = {Mathematika},
  FJOURNAL = {Mathematika. A Journal of Pure and Applied Mathematics},
    VOLUME = {61},
      YEAR = {2015},
    NUMBER = {2},
     PAGES = {414--443},
      ISSN = {0025-5793,2041-7942},
   MRCLASS = {46B70 (46E35)},
  MRNUMBER = {3343061},
MRREVIEWER = {Oscar\ Dom\'inguez},
       DOI = {10.1112/S0025579314000278},
}

@book {MR2328004,
    AUTHOR = {Tartar, Luc},
     TITLE = {An introduction to {S}obolev spaces and interpolation spaces},
    SERIES = {Lecture Notes of the Unione Matematica Italiana},
    VOLUME = {3},
 PUBLISHER = {Springer, Berlin; UMI, Bologna},
      YEAR = {2007},
     PAGES = {xxvi+218},
      ISBN = {978-3-540-71482-8; 3-540-71482-0},
   MRCLASS = {46E35 (41A05 46B70 46M35)},
  MRNUMBER = {2328004},
MRREVIEWER = {Joan\ L.\ Cerd\`a},
}

@book {MR1742312,
    AUTHOR = {McLean, William},
     TITLE = {Strongly elliptic systems and boundary integral equations},
 PUBLISHER = {Cambridge University Press, Cambridge},
      YEAR = {2000},
     PAGES = {xiv+357},
      ISBN = {0-521-66332-6; 0-521-66375-X},
   MRCLASS = {35J45 (47F05 47G10 47N20 65N38)},
  MRNUMBER = {1742312},
MRREVIEWER = {Dorina\ I.\ Mitrea},
}

@book {MR350177,
    AUTHOR = {Lions, J.-L. and Magenes, E.},
     TITLE = {Non-homogeneous boundary value problems and applications.
              {V}ol. {I}},
    SERIES = {Die Grundlehren der mathematischen Wissenschaften},
    VOLUME = {Band 181},
      NOTE = {Translated from the French by P. Kenneth},
 PUBLISHER = {Springer-Verlag, New York-Heidelberg},
      YEAR = {1972},
     PAGES = {xvi+357},
   MRCLASS = {35JXX (35KXX 35LXX 46E35)},
  MRNUMBER = {350177},
}

@article {MR2754080,
    AUTHOR = {Stinga, Pablo Ra\'ul and Torrea, Jos\'e{} Luis},
     TITLE = {Extension problem and {H}arnack's inequality for some
              fractional operators},
   JOURNAL = {Comm. Partial Differential Equations},
  FJOURNAL = {Communications in Partial Differential Equations},
    VOLUME = {35},
      YEAR = {2010},
    NUMBER = {11},
     PAGES = {2092--2122},
      ISSN = {0360-5302,1532-4133},
   MRCLASS = {35R11 (35B45 35B50 35B51 35B65)},
  MRNUMBER = {2754080},
MRREVIEWER = {Nasser-eddine\ Tatar},
       DOI = {10.1080/03605301003735680},
}

@article {MR2825595,
    AUTHOR = {Capella, Antonio and D\'avila, Juan and Dupaigne, Louis and
              Sire, Yannick},
     TITLE = {Regularity of radial extremal solutions for some non-local
              semilinear equations},
   JOURNAL = {Comm. Partial Differential Equations},
  FJOURNAL = {Communications in Partial Differential Equations},
    VOLUME = {36},
      YEAR = {2011},
    NUMBER = {8},
     PAGES = {1353--1384},
      ISSN = {0360-5302,1532-4133},
   MRCLASS = {35R11 (35B65 35D30 35J25 35J91 35S10)},
  MRNUMBER = {2825595},
MRREVIEWER = {William\ Margulies},
       DOI = {10.1080/03605302.2011.562954},
}

@article {MR2646117,
    AUTHOR = {Cabr\'e, Xavier and Tan, Jinggang},
     TITLE = {Positive solutions of nonlinear problems involving the square
              root of the {L}aplacian},
   JOURNAL = {Adv. Math.},
  FJOURNAL = {Advances in Mathematics},
    VOLUME = {224},
      YEAR = {2010},
    NUMBER = {5},
     PAGES = {2052--2093},
      ISSN = {0001-8708,1090-2082},
   MRCLASS = {35J10 (35S05)},
  MRNUMBER = {2646117},
MRREVIEWER = {B.\ Kellogg},
       DOI = {10.1016/j.aim.2010.01.025},
}

@article {MR3989717,
    AUTHOR = {Banjai, Lehel and Melenk, Jens M. and Nochetto, Ricardo H. and
              Ot\'arola, Enrique and Salgado, Abner J. and Schwab,
              Christoph},
     TITLE = {Tensor {FEM} for spectral fractional diffusion},
   JOURNAL = {Found. Comput. Math.},
  FJOURNAL = {Foundations of Computational Mathematics. The Journal of the
              Society for the Foundations of Computational Mathematics},
    VOLUME = {19},
      YEAR = {2019},
    NUMBER = {4},
     PAGES = {901--962},
      ISSN = {1615-3375,1615-3383},
   MRCLASS = {65N30 (26A33 35R11 65N12)},
  MRNUMBER = {3989717},
       DOI = {10.1007/s10208-018-9402-3},
}

@article {MR4505157,
    AUTHOR = {Byun, Sun-Sig and Song, Kyeong},
     TITLE = {Mixed local and nonlocal equations with measure data},
   JOURNAL = {Calc. Var. Partial Differential Equations},
  FJOURNAL = {Calculus of Variations and Partial Differential Equations},
    VOLUME = {62},
      YEAR = {2023},
    NUMBER = {1},
     PAGES = {Paper No. 14, 35},
      ISSN = {0944-2669,1432-0835},
   MRCLASS = {35B65 (31C45 35A01 35R06)},
  MRNUMBER = {4505157},
       DOI = {10.1007/s00526-022-02349-7},
}

@article {MR4510212,
    AUTHOR = {Antil, Harbir and Brown, Thomas S. and Verma, Deepanshu and
              Warma, Mahamadi},
     TITLE = {Optimal control of fractional {PDE}s with state and control
              constraints},
   JOURNAL = {Pure Appl. Funct. Anal.},
  FJOURNAL = {Pure and Applied Functional Analysis},
    VOLUME = {7},
      YEAR = {2022},
    NUMBER = {5},
     PAGES = {1533--1560},
      ISSN = {2189-3756,2189-3764},
   MRCLASS = {49J20 (35Q49 49K20 49M41 65R20)},
  MRNUMBER = {4510212},
}

@article {MR4026184,
    AUTHOR = {G\'omez-Castro, David and V\'azquez, Juan Luis},
     TITLE = {The fractional {S}chr\"odinger equation with singular
              potential and measure data},
   JOURNAL = {Discrete Contin. Dyn. Syst.},
  FJOURNAL = {Discrete and Continuous Dynamical Systems. Series A},
    VOLUME = {39},
      YEAR = {2019},
    NUMBER = {12},
     PAGES = {7113--7139},
      ISSN = {1078-0947,1553-5231},
   MRCLASS = {35R11 (35D30 35J75)},
  MRNUMBER = {4026184},
MRREVIEWER = {Vincenzo\ Ambrosio},
       DOI = {10.3934/dcds.2019298},
}

@article {MR3339179,
    AUTHOR = {Kuusi, Tuomo and Mingione, Giuseppe and Sire, Yannick},
     TITLE = {Nonlocal equations with measure data},
   JOURNAL = {Comm. Math. Phys.},
  FJOURNAL = {Communications in Mathematical Physics},
    VOLUME = {337},
      YEAR = {2015},
    NUMBER = {3},
     PAGES = {1317--1368},
      ISSN = {0010-3616,1432-0916},
   MRCLASS = {35R11 (31B35 35A01 35R06 35R09)},
  MRNUMBER = {3339179},
MRREVIEWER = {Paolo\ Baroni},
       DOI = {10.1007/s00220-015-2356-2},
}

@article {MR3217045,
    AUTHOR = {Chen, Huyuan and V\'eron, Laurent},
     TITLE = {Semilinear fractional elliptic equations involving measures},
   JOURNAL = {J. Differential Equations},
  FJOURNAL = {Journal of Differential Equations},
    VOLUME = {257},
      YEAR = {2014},
    NUMBER = {5},
     PAGES = {1457--1486},
      ISSN = {0022-0396,1090-2732},
   MRCLASS = {35R11 (35J08 35J61 35R06)},
  MRNUMBER = {3217045},
MRREVIEWER = {Zhuoran\ Du},
       DOI = {10.1016/j.jde.2014.05.012},
}

@article {MR2779579,
    AUTHOR = {Karlsen, Kenneth H. and Petitta, Francesco and Ulusoy,
              Suleyman},
     TITLE = {A duality approach to the fractional {L}aplacian with measure
              data},
   JOURNAL = {Publ. Mat.},
  FJOURNAL = {Publicacions Matem\`atiques},
    VOLUME = {55},
      YEAR = {2011},
    NUMBER = {1},
     PAGES = {151--161},
      ISSN = {0214-1493,2014-4350},
   MRCLASS = {35R11 (35A01 35A02 35B65 35J15 35R06)},
  MRNUMBER = {2779579},
MRREVIEWER = {M.\ Pilar\ Velasco},
       DOI = {10.5565/PUBLMAT\_55111\_07},
}

@article {MR3893441,
    AUTHOR = {Bonito, Andrea and Borthagaray, Juan Pablo and Nochetto,
              Ricardo H. and Ot\'arola, Enrique and Salgado, Abner J.},
     TITLE = {Numerical methods for fractional diffusion},
   JOURNAL = {Comput. Vis. Sci.},
  FJOURNAL = {Computing and Visualization in Science},
    VOLUME = {19},
      YEAR = {2018},
    NUMBER = {5-6},
     PAGES = {19--46},
      ISSN = {1432-9360,1433-0369},
   MRCLASS = {65N30 (26A33 65N15)},
  MRNUMBER = {3893441},
       DOI = {10.1007/s00791-018-0289-y},
}

@article {MR3348172,
    AUTHOR = {Nochetto, Ricardo H. and Ot\'arola, Enrique and Salgado, Abner
              J.},
     TITLE = {A {PDE} approach to fractional diffusion in general domains: a
              priori error analysis},
   JOURNAL = {Found. Comput. Math.},
  FJOURNAL = {Foundations of Computational Mathematics. The Journal of the
              Society for the Foundations of Computational Mathematics},
    VOLUME = {15},
      YEAR = {2015},
    NUMBER = {3},
     PAGES = {733--791},
      ISSN = {1615-3375,1615-3383},
   MRCLASS = {65N30 (35S11 65N12)},
  MRNUMBER = {3348172},
MRREVIEWER = {Patrick\ Henning},
       DOI = {10.1007/s10208-014-9208-x},
}

@book {MR609148,
    AUTHOR = {Birman, M. {\v{S}}. and Solomjak, M. Z.},
     TITLE = {{Spektral'naya teoriya samosopryazhennykh
              operatorov v gil'bertovom prostranstve}},
 PUBLISHER = {Leningrad. Univ., Leningrad},
      YEAR = {1980},
     PAGES = {264 pp. (loose errata)},
   MRCLASS = {47-01},
  MRNUMBER = {609148},
MRREVIEWER = {Dan\ Timotin},
}

@book {MR4269305,
    AUTHOR = {Ern, Alexandre and Guermond, Jean-Luc},
     TITLE = {Finite elements {II}---{G}alerkin approximation, elliptic and
              mixed {PDE}s},
    SERIES = {Texts in Applied Mathematics},
    VOLUME = {73},
 PUBLISHER = {Springer, Cham},
      YEAR = {[2021] \copyright 2021},
     PAGES = {ix+492},
      ISBN = {978-3-030-56922-8; 978-3-030-56923-5},
   MRCLASS = {65-01},
  MRNUMBER = {4269305},
       DOI = {10.1007/978-3-030-56923-5},
}

@article {MR3118443,
    AUTHOR = {Ciarlet, Jr., P.},
     TITLE = {Analysis of the {S}cott-{Z}hang interpolation in the
              fractional order {S}obolev spaces},
   JOURNAL = {J. Numer. Math.},
  FJOURNAL = {Journal of Numerical Mathematics},
    VOLUME = {21},
      YEAR = {2013},
    NUMBER = {3},
     PAGES = {173--180},
      ISSN = {1570-2820,1569-3953},
   MRCLASS = {65N30 (65N15)},
  MRNUMBER = {3118443},
MRREVIEWER = {Francesco\ Calabr\`o},
       DOI = {10.1515/jnum-2013-0007},
}

@article {MR4238777,
    AUTHOR = {Faustmann, Markus and Melenk, Jens Markus and Parvizi, Maryam},
     TITLE = {On the stability of {S}cott-{Z}hang type operators and
              application to multilevel preconditioning in fractional
              diffusion},
   JOURNAL = {ESAIM Math. Model. Numer. Anal.},
  FJOURNAL = {ESAIM. Mathematical Modelling and Numerical Analysis},
    VOLUME = {55},
      YEAR = {2021},
    NUMBER = {2},
     PAGES = {595--625},
      ISSN = {2822-7840,2804-7214},
   MRCLASS = {65N30 (35R11 65N22)},
  MRNUMBER = {4238777},
       DOI = {10.1051/m2an/2020079},
}

@article {MR1255054,
    AUTHOR = {Matsuki, Mihoko and Ushijima, Teruo},
     TITLE = {A note on the fractional powers of operators approximating a
              positive definite selfadjoint operator},
   JOURNAL = {J. Fac. Sci. Univ. Tokyo Sect. IA Math.},
  FJOURNAL = {Journal of the Faculty of Science. University of Tokyo.
              Section IA. Mathematics},
    VOLUME = {40},
      YEAR = {1993},
    NUMBER = {2},
     PAGES = {517--528},
      ISSN = {0040-8980},
   MRCLASS = {47B25 (47A60)},
  MRNUMBER = {1255054},
MRREVIEWER = {Alberto\ Venni},
}

@article{Sawyer,
      title={A {S}emi-{A}nalytic {D}iagonalization {FEM} for the {S}pectral {F}ractional {L}aplacian}, 
      author={Abner J. Salgado and Shane E. Sawyer},
      year={2025},
      note={(accepted) arXiv:2409.17388},
      journal ={SIAM J. Sci. Comput.},
}

@article {MR3429589,
    AUTHOR = {Hosseini, Bamdad and Nigam, Nilima and Stockie, John M.},
     TITLE = {On regularizations of the {D}irac delta distribution},
   JOURNAL = {J. Comput. Phys.},
  FJOURNAL = {Journal of Computational Physics},
    VOLUME = {305},
      YEAR = {2016},
     PAGES = {423--447},
      ISSN = {0021-9991,1090-2716},
   MRCLASS = {35A35 (35A08)},
  MRNUMBER = {3429589},
       DOI = {10.1016/j.jcp.2015.10.054},
}

@article {MR4169484,
    AUTHOR = {Heltai, Luca and Lei, Wenyu},
     TITLE = {A priori error estimates of regularized elliptic problems},
   JOURNAL = {Numer. Math.},
  FJOURNAL = {Numerische Mathematik},
    VOLUME = {146},
      YEAR = {2020},
    NUMBER = {3},
     PAGES = {571--596},
      ISSN = {0029-599X,0945-3245},
   MRCLASS = {65N30 (65N15)},
  MRNUMBER = {4169484},
MRREVIEWER = {Nicolae\ Pop},
       DOI = {10.1007/s00211-020-01152-w},
}

@article {MR337032,
    AUTHOR = {Scott, Ridgway},
     TITLE = {Finite element convergence for singular data},
   JOURNAL = {Numer. Math.},
  FJOURNAL = {Numerische Mathematik},
    VOLUME = {21},
      YEAR = {1973/74},
     PAGES = {317--327},
      ISSN = {0029-599X,0945-3245},
   MRCLASS = {65N35},
  MRNUMBER = {337032},
MRREVIEWER = {G.\ T.\ Symm},
       DOI = {10.1007/BF01436386},
}

@article{MR3043640,
  AUTHOR = {Hecht, F.},
  TITLE = {New development in {F}ree{F}em++},
  JOURNAL = {J. Numer. Math.},
  FJOURNAL = {Journal of Numerical Mathematics},
  VOLUME = {20}, YEAR = {2012},
  NUMBER = {3-4}, PAGES = {251--265},
  ISSN = {1570-2820},
  MRCLASS = {65Y15},
  MRNUMBER = {3043640},
  URL = {https://freefem.org/}
}

\end{document}